\documentclass{amsart}

\usepackage{amsfonts,amsthm,latexsym,amsmath,amssymb,amscd,amsmath, epsf}
\usepackage{graphicx}
\newtheorem{theorem}{Theorem}
\newtheorem{proposition}[theorem]{Proposition}
\newtheorem{lemma}[theorem]{Lemma}

\newtheorem{corollary}{Corollary}

\newtheorem{remark}{Remark}
\newtheorem{conjecture}{Conjecture}

\newtheorem{problem}{Problem}

\newcommand{\cDD}{\mathcal{D}}

\newcommand{\cP}{\mathcal{P}}

\newcommand{\bR}{\mathbb{R}}
\newcommand{\bC}{\mathbb{C}}
\newcommand{\ee}{\end{equation}}
\newcommand {\al}{\alpha}
\newcommand {\la}{\lambda}
\newcommand{\be}{\beta}
\newcommand {\ga}{\gamma}

\newcommand {\De}{\Delta}

\newcommand{\LP}{\mathcal {L-P}}
\newcommand {\q}{\widetilde q}
\newcommand {\fS} {\mathfrak S}
\newcommand {\bN}{\mathbb N}

\begin{document}
\title[Hardy-Petrovitch-Hutchinson's problem and partial theta function]{Hardy-Petrovitch-Hutchinson's problem and partial theta function}

\author[V.~Kostov]{Vladimir Petrov Kostov}
\address{Universit\'e de Nice, Laboratoire de Math\'ematiques, 
Parc Valrose, 06108 Nice Cedex 2, France}
\email{kostov@math.unice.fr}



\author[B. Shapiro]{Boris Shapiro}

\address{   Department of Mathematics,
            Stockholm University,
            S-10691, Stockholm, Sweden}
\email{shapiro@math.su.se}



\begin{abstract} In 1907 M.~Petrovitch \cite{Pe} initiated the study of a class of entire functions  all whose finite sections  are real-rooted polynomials. He was motivated by  previous studies \cite{La}  of E.~Laguerre on uniform limits of sequences of real-rooted polynomials and an interesting result of G.~H.~Hardy~\cite {Ha}. (The term 'section' meaning truncation although quite inappropriate seems to be  standard in the literature on this topic.)  An explicit description of this class in terms of the coefficients of a series is impossible since it is determined by an infinite number of determinantal inequalities one for each degree. However, interesting necessary or sufficient conditions can be formulated. In particular, J.~I.~Hutchinson \cite {Hu}  has shown that  an entire function  $p(x)=a_0+a_1x+...+a_nx^n+...$ with strictly positive coefficients has the property that any its finite segment $a_ix^i+a_{i+1}x^{i+1}+...+a_jx^{j}$ has all real roots if and only if for all $i=1,...,n-1$ one has $\frac{a_i^2}{a_{i-1}a_{i+1}}\ge 4$. In the present paper we give sharp lower bounds on the ratios $\frac{a_i^2}{a_{i-1}a_{i+1}},\;i=1,2, ... $ for the class considered by M.~Petrovitch.   In particular, we  show that the limit of these minima when $i\to \infty$ equals the inverse of the maximal positive value of the parameter for which the classical 
partial theta function belongs to the Laguerre-P\'olya class. We also explain the relation between Newton's and Hutchinson's inequalities and the logarithmic image of the set of all real-rooted polynomials with positive coefficients.

\end{abstract}
\dedicatory{To the memory of Vladimir Igorevich Arnold} 

\date{\today}
\keywords{real-rooted polynomias, section-hyperbolic polynomials, Laguerre-P\'olya class} 
\subjclass [2010] {Primary 30C15;  Secondary  12D10, 26C05}

\maketitle

\section{Introduction} 

In what follows we will use the terms 'real-rooted polynomial' and 'hyperbolic polynomial' as synonyms. 'Hyperbolic polynomials'  are  inherited from partial differential equations.    
Consider the space $\cP_n$ of polynomials of the form $p(x)=1+a_1x+...+a_nx^n$ with real coefficients. A polynomial $p(x)\in \cP_n$ with all positive coefficients is called {\em section-hyperbolic} if for all $i=1,...,n$ its section $1+a_1x+...+a_ix^i$ of degree $i$ is hyperbolic. 
Let $\De_n\subset \cP_n$ be the set of all section-hyperbolic polynomials of degree $n$ and let $\Delta=\bigcup_n\Delta_n$ be the set of all section-hyperbolic polynomials. (Notice that by a result of E.~Laguerre \cite {La} 
a formal power series whose sections belong to $\Delta$ is an entire function lying in the Laguerre-P\'olya class.)   We will call an entire function with positive coefficients all whose sections belong to $\Delta$ {\em section-hyperbolic}. The  following question was posed to the second author in April 2010 by Professors O.~Katkova and A.~Vishnyakova   who  attributed it to Professor I.~V.~Ostrovskii, comp. \cite{Ost}, \cite {Zh}.  

\begin{problem}[Hardy-Hutchinson-Petrovitch-Ostrovskii\footnote{The order of the names is chronological in accordance with the years of their contributions to the topic under consideration. According to late Vladimir Igorevich Arnold (whose students we both were)  if a mathematical concept is named after some mathematician this  in most of the cases is due to the fact that this person has {\bf never}Ê considered this particular concept.}]

For a given positive integer $i$ find or estimate  
$$m_i=inf_{p\in \De} \frac{a_i^2}{a_{i-1}a_{i+1}}.$$
\end{problem}Ê

Denote by $Pol_n$ the space of all monic real polynomials of degree $n$ and by $\Sigma_n\subset Pol_n$ the set of all such polynomials having all real and negative roots. Given a polynomial 
$p(x)=a_0+a_1x+...+a_nx^n,\; a_n\neq 0$  define its reverted polynomial  as $P(x)=a_0x^n+a_1x^{n-1}+...+a_n=x^np(1/x)$.  
Obviously,  reversion  is a diffeomorphism between $\cP_n$ and $Pol_n$. We say that a polynomial $P(x)=x^n+a_1x^{n-1}+...+a_n$ is {\em reverted section-hyperbolic} if its sections $x^n+...+a_ix^{n-i}$ from the back are hyperbolic for all $i=1,...,n$. 
Reversion  sends diffeomorphically the set $\Delta_n\subset \cP_n$ of all section-hyperbolic polynomials onto the set of all reverted section-hyperbolic polynomials. By a slight abuse of notation we denote the latter set by $\Delta_n\subset Pol_n$ and will freely use both interpretations. (In fact, the second interpretation can be already found in the original paper \cite{Pe}.) 
Observe that the quantities $\frac{a_i^2}{a_{i-1}a_{i+1}}$ 
are preserved by  reversion  (up to a change of indices). 

Notice that since the natural projection $\pi_i: \Delta_{i+1}\to \Delta_i$  'forgetting' the leading monomial is surjective one has 
$$m_i=inf_{p\in \De} \frac{a_i^2}{a_{i-1}a_{i+1}}=inf_{p\in \De_{i+1}} \frac{a_i^2}{a_{i-1}a_{i+1}},$$
i.e., to determine $m_i$ it suffices to consider $\Delta_{i+1}$.  
Moreover, by Hutchinson's theorem, see \cite{Hu}, one has $m_i\le 4$. (Hutchinson's theorem was rediscovered 70 years later in \cite{Ku}.) Apparently Petrovitch knew that 
$m_1=4,\;m_2=\frac{27}{8}=3.375$ and $m_3\approx 3.264$, see pp. 42--43 of \cite {Pe} but he made no explicit claim.  On p.~331 of \cite{Hu}  Hutchinson writes  that the sequence  $\{m_i\}$ should be strictly decreasing to some unknown limit $m_\infty$ of which it is known that it should exceed $ 2$. 

The next  result was proved by O.~Katkova and A.~Visnyakova around 2006. 

\begin{theorem}\label{th:est} For a given positive integer $i$ one has 
$$m_i\ge 3.$$
\end{theorem} 

\begin{remark}\label{rm:1} \rm {Observe that the quantities $\frac{a_i^2}{a_{i-1}a_{i+1}}$ can attain arbitrarily large values on $\De_n$ so it is their minimal values which are important. Also notice that  if one instead of $\De_n$ considers these quantities on the set $\Sigma_n$ of all monic polynomials with negative roots, then the famous Newton's inequalities claim that for any given $i=1,...,n-1$ one has 
\begin{equation}\label{eq:Newton}
\frac{a_i^2}{a_{i-1}a_{i+1}}\ge \frac{(n-i+1)(i+1)}{(n-i)i}>1, 
\end{equation} 
see e.g. \cite {Ni} and comp.  Proposition~\ref{prop:log. convexity}.
Moreover, the equality in \eqref{eq:Newton}  is attained exactly on polynomials with all real and coinciding roots. In particular, for appropriate choices of $n$ and $i$  the quantity ${a_i^2}/{a_{i-1}a_{i+1}}$ can be arbitrarily close to $1$.} 

\end{remark} 

Below we  solve the problem of Hardy-Petrovitch-Hutchinson-Ostrovskii by presenting 'explicitly'  the entire function which simultaneously  realizes all the above $m_i$. Namely, consider the following sequence of polynomials 
$p_1(x)=1+x,\; p_2(x)=1+x+x^2/4,\; p_3(x)=1+x+x^2/4+x^3/54,\; p_4(x)=1+x+x^2/4+x^3/54+(69-11\sqrt{33})x^4/13824, ...$ given by the following inductive procedure
$p_n(x)=p_{n-1}(x)+A_nx^n$, where $A_n$ is the maximal positive number such that $P_n(x)$ is hyperbolic. Denote by $p_\infty(x)=1+x+\sum_{n=2}^\infty A_nx^n$. (This series appears on p. 42 of \cite {Pe} and one can show that $p_\infty$ is an entire function.) 

\medskip  
The first result of this paper is as follows. 

\begin{theorem}Ê\label{th:main} For any positive integer $i$ one has 
$$m_i=inf_{p\in \De} \frac{a_i^2}{a_{i-1}a_{i+1}}= \frac{A_i^2}{A_{i-1}A_{i+1}}\;, $$
where $A_i$ are the above coefficients.  In other words, $m_i$ is attained at $p_{i+1}\in \Delta_{i+1}$ or, equivalently, the function $p_\infty(x)$ minimizes all $m_i$ simultaneously.  Moreover, $p_{i+1} $ is the unique (up to a scaling of the independent variable $x$) polynomial in $\Delta_{i+1}$  
minimizing  the quantity ${a_i^2}/{a_{i-1}a_{i+1}}$. 
\end{theorem}

By the above remark one can also conclude that $\frac{a_i^2}{a_{i-1}a_{i+1}}$ attains its minimum  at the monic polynomial $P_{i+1}\in Pol_{i+1}$ which is the reverted polynomial to $p_{i+1}$.   We call the inequalities of the form 
\begin{equation}\label{eq:Petrov}
\frac{a_i^2}{a_{i-1}a_{i+1}}\ge m_i
\end{equation} 
with $m_i$ defined above {\em Petrovitch's inequalities} and the inequalities of the form 
\begin{equation}\label{eq:Hutch}
\frac{a_i^2}{a_{i-1}a_{i+1}}\ge 4
\end{equation}
 {\em Hutchinson's inequalities}, see footnote above.  
Using Theorem~\ref{th:est} one can show  that $m_i$ are algebraic numbers and  calculate them  on computer with  an arbitrary  precision.  Namely, the 10 decimal places of the first 17 $m_i$'s are as follows: $m_1=4,\quad m_2= \frac {27}{8},\quad m_3= \frac{2(69 + 11 \sqrt {33})}{81}\approx 3.2639552867,\quad m_4\approx 3.2403064116,\quad m_5\approx 3.2351101647,\quad 
m_6\approx 3.2339623707,\quad$ $ m_7\approx 3.2337086596,\quad$ $ m_8\approx 3.2336525783,\quad$ $ 
m_9\approx 3.2336401824,\quad$ $m_{10}\approx 3.2336374426,\quad$ $ m_{11}\approx 3.2336368370\quad m_{12}\approx 3.2336367032,\quad m_{13}\approx 3.2336366736,\quad$ $ m_{14}\approx 3.2336366671,\quad$ 
$m_{15}\approx 3.2336366656,\quad$ $ m_{16}\approx 3.2336366653,\quad$ $ m_{17}\approx 3.2336366652.$ 
Notice that for all $i> 17$ all $m_i$ have their first 10 decimal places coinciding with that of $m_{17}$. Further calculations show that the next ten $m_i$ have the same 10 decimal places as $m_{17}$, they are monotone decreasing, and every second time the next decimal position stabilizes. Our next result confirms this behavior. 
\medskip

\begin{theorem}\label{th:monot} The sequence $\{m_i\}, i=1,2,...$ is strictly monotone decreasing. 
\end{theorem} 

From Theorems~\ref{th:est} and ~\ref{th:monot} Ê we get that  the sequence $\{m_i\}$ has a limit which we denote by $m_\infty$.  The main result of this paper is the explicit description of $m_\infty$ and the related entire function.  To do this we define the  formal power series $\Psi(q,u)$ in the variables $(q,u)$ which we by a small abuse of notation call the {\em partial theta function} 
\begin{equation}\label{eq:Psi}
\Psi(q,u)=\sum_{j=0}^\infty q^{\binom {j+1}{2}}u^j.
\end{equation} 
This function already appears on p. 330 of \cite{Hu}. 

\begin{remark} \rm {The standard partial theta function is usually defined  by the series $\Theta(q,u)=\sum_{j=0}^\infty(-1)^jq^{\binom{j}{2}}u^j$, see e.g. \cite{An1}, \cite{An2}, \cite{Ra} and these two functions satisfy the obvious relation
\begin{equation}\label{eq:psitheta}
\Psi(q,u)=\Theta(q,-qu)
\end{equation}
which allows to translate their properties to one another. 
  A number of  beautiful identities which it satisfies was stated without proofs  in Ramanujan's  "lost" notebook which was found by  G.~E.~Andrews after 50 years and who put  significant effort in proving these identities.  New results about the sum and product of partial theta functions can be found in e.g. \cite{AnWa}. It is also of interest in statistical physics and combinatorics, see \cite{So}. }
\end{remark}

In what follows we consider $q$ as a parameter and $u$ as the main variable. 
One can  easily see that $\Psi(q,u)$ has a positive radius of convergence as a function of $u$ if and only if  $|q| \le 1$. If $|q|=1$, then $\Psi(q,u)$ has a  convergence radius equal to $1$ while  for any such $q$ with $|q|<1$  the function $\Psi(q,u)$ is  entire. Moreover for small positive $q$ the series $\Psi(q,u)$ considered as a function in $u$ belongs to the Laguerre-P\'olya class $\LP^+$, i.e. it has   all real and negative roots, see e.g. \cite{Le}, Ch. 8.  (The well-known  characterization of these functions was obtained almost hundred years ago in \cite{PS}.) Notice that for $\Psi(q,u)$ the quotient $a_j^2/a_{j-1}a_{j+1}=\left(q^{\binom {j}{2}}\right)^2/q^{\binom{j}{2}}q^{\binom {j+2}{2}}$ equals ${1}/{q}$. 

Recall that  $P_i$ denotes the reverted polynomial  for $p_i$ introduced above.  (Since the constant term of $p_i$ equals $1$ one gets that  $P_i$ is monic.) In Lemma~\ref{lm:6} below we show that each $P_i$ has all simple negative roots except for a single double root which has the minimal absolute value  among all roots of $P_i$. Since each  root of $p_i$ is the  inverse of the corresponding root of $P_i$ one gets that $p_i$ also has all real and negative roots except for a single double root which has the maximal absolute value among the roots of $p_i$. Denote by 
$\zeta_i$ the unique double root of  $P_i$. For a positive integer $i=1,2,...$ define the {\em scaled reverted polynomial}  Ê$\widetilde P_i(x)=P_i(-\zeta_{i}x)/P_i(0)$.  The scaling of $P_i$ is done in such a way that its double root is placed at $-1$ and its constant term equals $1$.  The main result of our paper is as follows. 

\medskip
\begin{theorem}\label{th:limit} The limit $m_\infty=\lim_{i\to \infty} m_i$ exists and  coincides with ${1}/{\q}$, where $\q>0$ is the maximal positive number for which the series $\Psi(q,u)$ belongs to $\LP^+$ as a function in $u$. Moreover, the sequence $\{\widetilde P_i\}$ of the scaled reverted polynomials converges  to $\Psi(\q ,-\tilde u x)$, where $\tilde u$ is the unique real double root of $\Psi(\tilde q,u)$, see Figure~\ref{fig2}. 
 \end{theorem}

Notice that by \eqref{eq:psitheta} the function $\Theta(q,u)$ belongs to $\LP^+$ exactly on the same interval of values of $q$ as $\Psi(q,u)$ does, namely, $q\in (0,\q)$. The approximative value of $\q$  with 10 decimal places is $0.3092493386$, see Figure~\ref{fig2}. We will later show that $\tilde q$ is a root of transcendental equation~\eqref{eq:main}. Notice that the constant $1/\q\approx 3.2336$ has earlier appeared in the papers \cite{KLV1} and \cite{KLV2}, where the authors studied functions closely related to the above partial theta function.  Namely, one of  the main objects in paper~\cite{KLV1}  is the function 
$$g_a(x):=\sum_{j=0}^\infty \frac{x^k}{a^{k^2}},\; a>1.$$
Theorem~4 of \cite{KLV1} claims  that $g_a(x)$ has all hyperbolic sections, i.e. is section-hyperbolic  if and only if $a^2\ge 1/\q$.  (Similar statements can be found in a recent preprint \cite{Han}.)   
Again, notice that there exists a simple relation between $g_a(x)$ and $\Theta(q,u)$, namely
$$g_{\sqrt q 
}(u)=\theta(q,\sqrt{q}u)$$
and, therefore, the functions $g_{\sqrt q}(u)$, $\Psi(q,u)$ and  $\Theta(q,u)$ belong to $\LP^+$ exactly on the same interval $(0,\q)$ of values of $q$.  

Furthermore,  Theorem~2 of \cite{KLV2}  claims the following. 

\begin{theorem}\label{th:OlyaAnya}
Let $f(x)=\sum_{j=0}^\infty a_jx^j,\; a_j>0$ be an entire function and $S_n(x)=\sum_{j=0}^na_jx^j$ be its sections. Suppose that 
there exists a subsequence $\{n_j\}_{j=1}^\infty\subset \bN$ such that $S_{n_j}(x)$ is hyperbolic for $j=1,2,...$. If there exists $\delta_\infty(f)=\lim_{n\to\infty}\frac{a_{n}^2}{a_{n-1}a_{n+1}}$, then for any positive integer $m$ one has that  $\sum_{j=0}^m\frac{x^j}{j!(\sqrt{\delta_\infty})^{j^2}}$ is hyperbolic.
\end{theorem}

As a corollary of Theorem~\ref{th:OlyaAnya} one obtains that if $f(x)$  is section-hyperbolic (at least for all sufficiently large $n$) and $\delta_\infty(f)$ exists, then $\delta_\infty(f)\ge 1/\q$. 

\begin{figure}

\begin{center}
\includegraphics[scale=0.35]{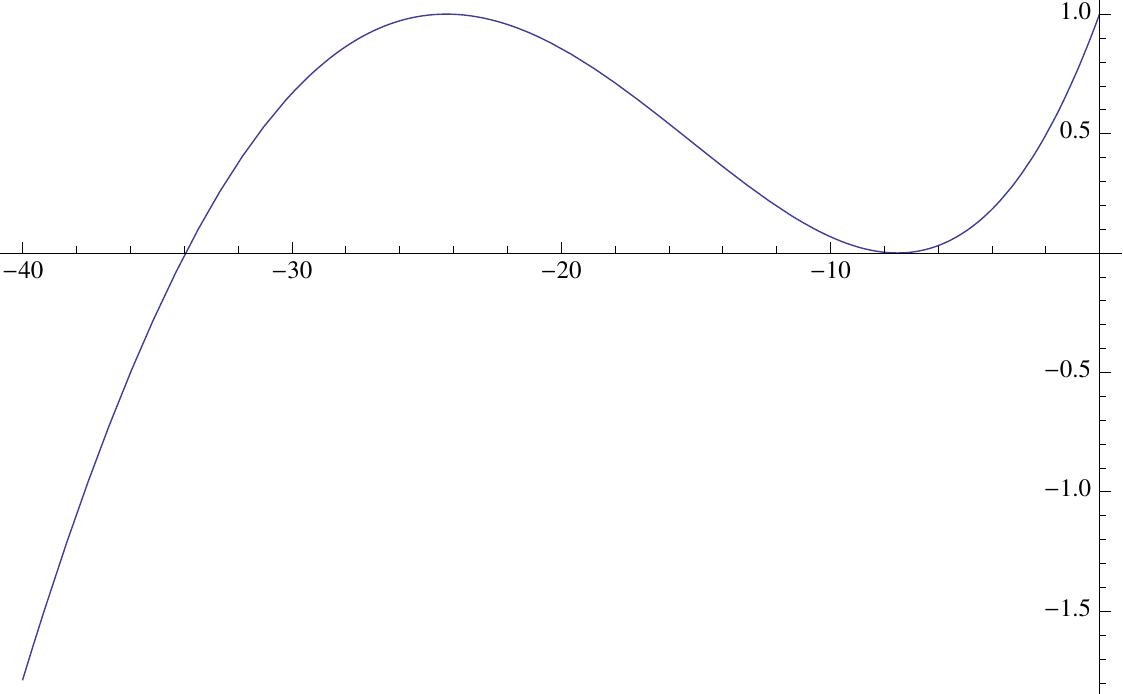} \hskip 1cm \includegraphics[scale=0.35]{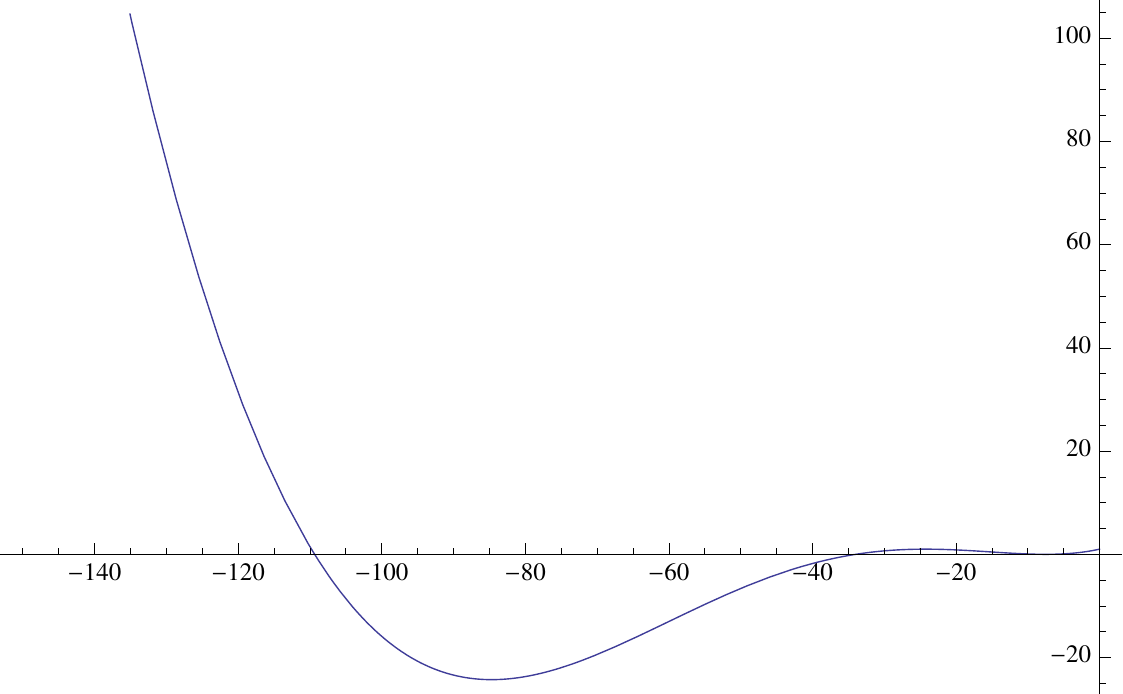}
\end{center}

\vskip 0.3cm

\caption {$\Psi(\q,u)$ in the intervals $[-40,0]$ and $[-150,0]$. (The value $\widetilde u$ of the negative double root of $\Psi(\q,u)$ with 10 decimal places is  $-7.5032559833$.)}
\label{fig2}
\end{figure}

Finally let us  mention a technical result of independent interest  very much in the spirit of the modern study of amoebas of complex hupersurfaces and, in particular, of discrimininants. 
 Denote   by $L\Sigma_n\subset \bR^n$ (respectively $L\Delta_n\subset \bR^n$)  the images of $\Sigma_n\subset Pol_n$ (respectively of $\Delta_n\subset Pol_n$) under taking coefficientwise logarithms.  

\begin{proposition}\label{prop:log. convexity} {\rm(i)} The  polyhedral cone given by Hutchinson's inequalities \eqref{eq:Hutch} (i.e. coinciding with the logarithmic image of the set of two-sided truncation hyperbolic polynomials) is the  maximal polyhedral cone contained in  $L\De_n$. The same cone is the maximal polyhedral cone contained  in $L\Sigma_n$. 

\noindent
{\rm(ii)} The minimal polyhedral cone containing  $L\De_k$ is given by Petrovitch's inequalities \eqref{eq:Petrov}Ê  while the minimal polyhedral cone containing $L\Sigma_n$ is given by Newton's inequalities \eqref{eq:Newton}, see Remark~\ref{rm:1}. 
\end{proposition}  

Notice that Hutchinson's cone is, on the other hand, the recession cone of the logarithmic image of the set  sign-invariant hyperbolic polynomials, see \cite{PRS}.  A  fact similar to Proposition~\ref{prop:log. convexity} is proven in Theorem~F of \cite{KLV2}.


\medskip
\noindent
{\it Acknowledments.}  The  authors want to thank Professors O.~Katkova and A.~Vishnyakova of Kharkov National University for the formulation of the problem and their proof of Theorem~1 which they kindly allowed us to include in the present paper. They also gave us important hints and contributed to the proof of Lemma~\ref{lm:finite}.  We are sincerely grateful to Professors A.~Eremenko and A.~Sokal and, especially, Professor G.~E.~Andrews for the valuable information about partial theta functions.  The second author  wants to acknowledge the hospitality of  Laboratoire de Math\'ematiques, Universit\'e de Nice  during his  visit 
in  April-May 2011 when this project was carried out. 

\medskip

\section{Proofs}Ê

We start with Theorem~\ref{th:est}. 

\begin{proof}

Take some polynomial $r_n(x)=a_0+a_1x+...+a_nx^n$ belonging to $\cP_n$ and 
set  $r_k(x)=a_0 + a_1 x + ...+a_k x^k, \  k=2, 3, \ldots
,n.$ By our assumption the polynomials $r_k(x)$ are hyperbolic for all
$k$. Set $\gamma_k:=\frac{a_{k-1}}{a_k},
\delta_k:=\frac{\gamma_k}{\gamma_{k-1}} =
\frac{a_{k-1}^2}{a_{k-2}a_{k}}.$

Let us fix an arbitrary $k= 3, 4, \ldots , n.$ Denote by
$0>x_1^{(k)}\ge x_2^{(k)}\ge ... \ge x_k^{(k)}$ the zeros of $r_k(x).$
Using the Cauchy inequality with $|x_j^{(k)}|^{1/2}$ and
$|x_j^{(k)}|^{3/2}$ we get

$$(x_1^{(k)}+x_2^{(k)}+...+x_k^{(k)}) ((x_1^{(k)})^3+(x_2^{(k)})^3+...+(x_k^{(k)})^3) \geq ((x_1^{(k)})^2+(x_2^{(k)})^2+...+(x_k^{(k)})^2)^2.$$

From the standard identities for (elementary) symmetric functions
we get
$$x_1^{(k)}+x_2^{(k)}+...+x_k^{(k)} = - \frac{a_{k-1}}{a_k} = -\gamma_k,$$
$$(x_1^{(k)})^2+(x_2^{(k)})^2+...+(x_k^{(k)})^2 = \left(\frac{a_{k-1}}{a_k}\right)^2 - 2
\frac{a_{k-2}}{a_k}=  \gamma_k^2 - 2 \gamma_k \gamma_{k-1},$$

$$(x_1^{(k)})^3+(x_2^{(k)})^3+...+(x_k^{(k)})^3 =  - \left(\frac{a_{k-1}}{a_k}\right)^3 +
 3 \frac{a_{k-1}}{a_k}\frac{a_{k-2}}{a_k} -3 \frac{a_{k-3}}{a_k} =
  - \gamma_k^3 + 3 \gamma_k^2 \gamma_{k-1} -3 \gamma_k \gamma_{k-1}\gamma_{k-2}.$$

Substituting these identities in the above inequality and dividing
by $\gamma_k^2 \gamma_{k-1}$ we obtain $$ \gamma_k - 4
\gamma_{k-1} + 3 \gamma_{k-2} \geq  0.$$

Dividing the latter inequality  by $\gamma_{k-2}$ and using
$\delta_j$'s we get the following inequality:

\begin{equation}\label{eq:2}
\delta_k \delta_{k-1} -4 \delta_{k-1} + 3 \geq 0.
\end{equation}

Since $k= 3, 4, \ldots , n$ is an arbitrary index, we get from
(\ref{eq:2}) the following system of inequalities

\begin{equation}\label{eq:2222}
\delta_k \delta_{k-1} -4 \delta_{k-1} + 3 \geq 0, \quad k= 3, 4,
\ldots , n.
\end{equation}
Since $r_2(x)$ is hyperbolic we have $\delta_2 \ge 4.$ Suppose
that the statement of the theorem is not true, and denote by $j$
the smallest index such that $\delta_j <3$, so that $\delta_{j-1}
\geq 3$ and $\delta_{j} < 3$ ($j=3, 4, \ldots, n$). We rewrite
(\ref{eq:2222}) for $k=j$ in the form
$$(\delta_j -4) \delta_{j-1} + 3 \geq 0 .$$
Since $\delta_j -4 <0 $, and $\delta_{j} < \delta_{j-1}$, the
above inequality implies
$$(\delta_j -4) \delta_{j} + 3 > 0 ,$$
whence $\delta_j \in (-\infty, 1) \bigcup (3, +\infty).$ By our
assumption $r_j(x)$ is a hyperbolic polynomial, thus $\delta_j \in
(-\infty, 1)$ is impossible. We conclude that $\delta_j\ge 3$.

\end{proof}

To  prove Theorem~\ref{th:main} we need some preliminaries. Observe that rescaling of the independent variable $x$ by an arbitrary positive constant acts on all spaces of polynomials we introduced above preserving the quantities $a_i^2/a_{n-1}a_n$. This action allows us to normalize  $a_1=1$ in $Pol_n$ and analogously $a_1=1$ in $\cP_n$ and, therefore, to reduce the number of parameters by  one. Define $\cP^1$ as the space of polynomials of the form $p(x)=1+x+a_2x^2+...+a_nx^n$ and $Pol_n^1$ as the space of polynomials of the form $P(x)=x^n+x^{n-1}+a_2x^{n-2}+...+a_n$. Notice that reversion  sends $\cP^1_n$ onto $Pol^1_n$ and that our main polynomials $p_i$ belong to $\cP_i^1$ while their reverted polynomials $P_i$ belong to $Pol^1_i$. From now on instead of working in $\De_n\subset Pol_n$ we will work in $\De_n^1\subset Pol_n^1$ which is the restriction of $\De_n$ to $Pol^1_n$. The above group action carries our proofs from one space to the other. 

Define the standard embedding $em_{j,n}: Pol_j^1\to Pol_n^1,\; j<n$ (respectively $Pol_j \to Pol_n$) given by multiplication of a monic polynomial of degree $j<n$ by $x^{n-j}$.  Obviously, the image  $em_{j,n}(Pol_j^1)\subset Pol_n^1$ coincides with the coordinate subspace of all monic polynomials having all  coefficients of degree less than $n-j$ vanishing.  Denote by $\cDD_j\subset Pol_j^1$ the standard discriminant consisting of all monic polynomials  of degree $j$  having at least one  real root of multiplicity at least $2$. Embedding $\cDD_j$ into $Pol_n^1$ using $em_{i,j}$ let us define the discriminant $\cDD_{j,n}\subset Pol_n^1$  by taking the trivial $(n-j)$-dimensional cylinder over $em_{i,j}(\cDD_j)$  along all coefficients of degree less than $(n-j)$.  Define $\Delta_n^1\subset Pol_n^1$ and $\Sigma_n^1\subset Pol_n^1$ as the restrictions of $\De_n$ and $\Sigma_n$ to $Pol_n^1$. Finally, consider the closure $\overline\De_n^1\subset Pol_n^1$ of the set $\De_n^1\subset Pol_n^1$. 

\begin{lemma}\label{lm:simplex}
{\rm (i)} The set $\overline \De_n^1$ has a natural stratification of an $(n-1)$-dimensional  simplex with vertices at $P_1=x^n+x^{n-1}, P_2=x^n+x^{n-1}+x^{n-2}/4,P_3, ....,P_n$. Different  $(n-2)$-dimensional (boundary) faces of $\overline \De_n^1$  belong to different $\cDD_{j,n},\; j=0,1,2,...,n$, see Figure~\ref{fig0}.  

\noindent
{\rm (ii)}Ê The  natural projection $\pi_n$  by 'forgetting' the constant term sends $\overline \De_n^1$ onto  $\overline\De_{n-1}^1$. 

\noindent
{\rm (iii)}Ê Any polynomial in $\overline \De_n^1$ can be connected to $P_n$ by a smooth path along which all coefficients are non-decreasing. 
\end{lemma}

\begin{remark}\rm{The original set $\De_n$ (respectively $\overline \De_n$) is the cylinder over $\De_n^1$ (respectively $\overline \De_n$) obtained by the action of the group of rescaling of $x$ by  positive constants.}
\end{remark}

\begin{proof}[Proof of Lemma~\ref{lm:simplex}]
The first two statements are rather obvious and proved by induction. $\overline \De_{n-1}^1$ is naturally embedded in the hyperplane $a_n=0$ of $Pol_n^1$ using the multiplication of polynomials of degree $n-1$ by $x$. Then  $\overline \De_{n}^1$ is fibered over the image of $\overline \De_{n-1}^1$ in $Pol_n^1$ along the constant term. To prove the third statement we show that the face $\cDD_n$ of the boundary of $\overline \De_{n}^1$ can be expressed as $a_n=a_n(a_2,....,a_{n-1})$, where $(a_2,...,a_{n-1})\in \overline \De_{n-1}^1$.  Moreover, for each $i=2,...,n-1$ one has $\partial a_n/\partial a_i>0$ in the whole open $\De_{n-1}^1$.   Indeed, denote by $x_i<0$ the roots of $P(x)$.   The double root is denoted by  $x_{n-1}=x_n$. Since  $a_{i}>0$ for all $i=1,...,n$  it 
will be convenient to consider them as elementary symmetric functions  in  the positive quantities $-x_j$. 
One has $\partial a_n/\partial a_{i}=
\sum _{j=1}^{n-1}(\partial a_n/\partial (-x_j))/(\partial a_{i}/\partial (-x_j))$. 
The quantities $a_n$, $\partial a_n/\partial (-x_j)$ and $\partial a_{i}/\partial (-x_j)$ 
are given by  homogeneous polynomials with positive coefficients in  all $-x_j$. This fact implies that the direction derivative of the function $a_n$ is non-negative along any vector in  $\De_{n-1}^1$ with all non-negative coordinates.  Using this statement together with induction on $n$ we get that any polynomial in $\De_n^1$ can be connected to $P_n$ by a smooth path with non-decreasing  coordinates. 
\end{proof} 

\begin{remark} \rm{In Lemma~\ref{lm:6} below we will prove that each polynomial in $\De_n^1$ has either simple negative roots or at most one  double root (in which case it belongs to $\cDD_n$) which is the rightmost among all roots of the considered polynomial. } 
\end{remark}

\begin{figure}

\begin{center}
\includegraphics[scale=0.35]{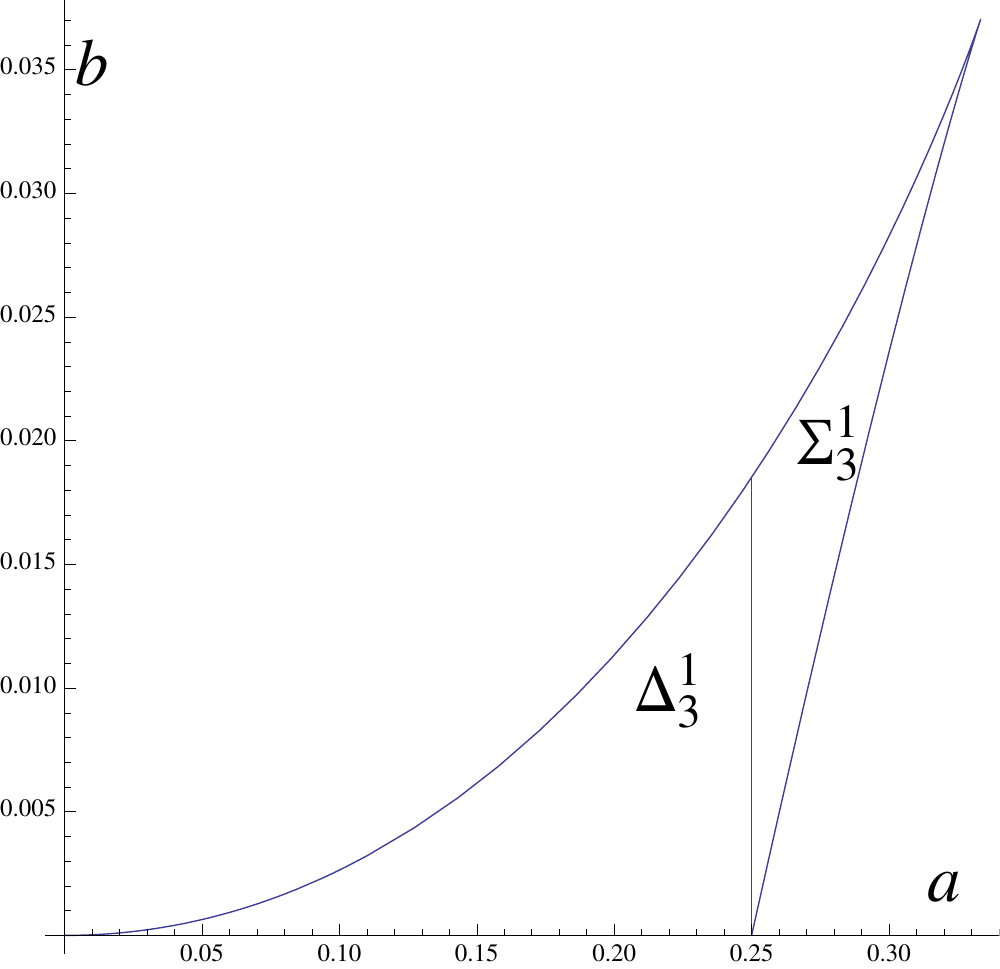}
\end{center}

\vskip 0.3cm

\caption {The domains $\Sigma_3^1$ and $\Delta_3^1$ for the family $x^3+x^2+ax+b$. (Notice that $\Sigma_3^1$  is the largest curvilinear triangle containing $\Delta_3^1$.)}
\label{fig0}
\end{figure}

\begin{proof} [Proof of Theorem~\ref{th:main}] 
Given some reverted section-hyperbolic polynomial $P(x)=x^n+x^{n-1}+a_2x^2+...+a_n$ of degree $n\geq 4$ consider the function $\kappa:=a_{n-1}^2/a_{n-2}a_n$. We want to show that $m_{n-1}=\min_{\De_n^1}\kappa$ is attained at $P_n(x)$ which is the reverted polynomial to  $p_n(x)$ defined in the Introduction.  
For fixed $a_2,..., a_{n-1}$ the function $\kappa$  is minimal when $a_n$ is maximal  in which case the polynomial $P(x)$ belongs to $\cDD_n$. Thus, we can restrict our consideration to $\De_n^1\ni P(x)\in \cDD_n$. Since each $P(x)$ can be connected to  $P_n(x)$ by a smooth 
path  along which each coefficient is non-decreasing it is enough to show that for all $i=2,...,n-1$ the partial derivative $\partial \kappa/\partial a_i$ is negative when $\kappa$ is restricted to $\cDD_n$. There are three different  cases to consider: (1) $i< n-2$; (2) $i=n-2$ and (3)  $i=n-1$. For $i<n-2$ one has 
$$\frac{\partial \kappa}{\partial a_i}=\frac{\partial \kappa}{\partial a_n}\frac{\partial a_n}{\partial a_i}=-\frac{a_{n-1}^2}{a_{n-2}a_n^2}\frac{\partial a_n}{\partial a_i}.$$
Since all $a_i>0$ and ${\partial a_n}/{\partial a_i}>0$ on $\cDD_n$ by Lemma~\ref{lm:simplex}  one has that case (1) is settled. 
Analogously, we have  
$$\frac{\partial \kappa}{\partial a_{n-2}}=-a_{n-1}^2\frac{a_n+a_{n-2}(\partial a_n/\partial a_{n-2})}{(a_{n-2}a_n)^2}~.$$
Again since all $a_i>0$ and ${\partial a_n}/{\partial a_{n-2}}>0$ on $\cDD_n$   case (2) is settled. 

Finally, one has 
$$\frac{\partial \kappa}{\partial a_{n-1}}=\frac{2a_{n-1}}{a_{n-2}a_n}-\frac{a_{n-1}^2(\partial a_n/\partial a_{n-1})}{a_{n-2}a_n^2}=
\frac{a_{n-1}(2a_n-a_{n-1}(\partial a_n/\partial a_{n-1}))}{(a_{n-2}a_n^2)}~~,~{\rm where}$$

$$a_{n-1}\frac{\partial a_n}{\partial a_{n-1}}=
a_{n-1}\sum _{i=1}^{n-1}\frac{\partial a_n}{\partial (-x_i)}\left/\frac{\partial a_{n-1}}{\partial (-x_i)}\right..$$
For $i\neq n-1$ one has $a_{n-1}=f_i+h_ig_i$, where $f_i$ and $g_i$ are homogeneous 
polynomials with positive coefficients depending on the positive 
variables $-x_k,\; k\neq i$. Therefore for $i\neq n-1$ one has 

\begin{equation}\label{neq}
a_{n-1}>-x_ig_i=-x_i\frac{\partial a_{n-1}}{\partial (-x_i)}~~{\quad\rm and\quad}~~ 
a_{n-1}\frac{\partial a_n}{\partial (-x_i)}\left/\frac{\partial a_{n-2}}{\partial (-x_i)}>-x_i\frac{\partial a_n}{\partial (-x_i)}\right..
\end{equation} 

For $i=n-1$ one has $a_{n-1}=-x_{n-1}v+x_{n-1}^2w$ and 
$$-x_{n-1}\frac{\partial a_{n-1}}{\partial (-x_{n-1})}=-x_{n-1}v+2x_{n-1}^2w.$$ 
Here $v$ and $w$ are given by homogeneous polynomials with positive coefficients in $-x_j$. (In what follows we will not need their explicit formulas.) 
Therefore, one has 
$$a_{n-1}>-\frac{x_{n-1}}{2}\frac{\partial a_{n-1}}{\partial (-x_{n-1})}.$$ 
Thus 

$$2a_n-a_{n-1}\frac{\partial a_n}{\partial a_{n-1}}<2a_n+\frac{1}{2}\sum _{i=1}^{n-1}x_i\frac{\partial a_n}{\partial (-x_i)}
<a_n\left(2-\frac{n}{2}\right)\leq 0~.$$
For the homogeneous  polynomial $a_n$ of degree $n$ we used Euler's identity  
$$na_n=-\sum _{i=1}^{n-1}x_i\frac{\partial a_n}{\partial (-x_i)}.$$

By the above argument  any directional derivative $\partial \kappa/\partial {\vec u}$ is non-positive if  $\vec{u}$ is an arbitrary  vector in $\De_{n-1}^1$ with all non-negative coordinates. Moreover since any polynomial $P\in  \Delta_n^1$ can be connected with $P_n$ by a smooth path with nondecreasing (and on some subintervals strictly increasing) coordinates we have that the value of $\kappa$ at $P_n$ is strictly smaller at any other such $P\neq P_n$. In the bigger set $\Delta_n$ this means that only polynomials obtained from $P_n$ by scaling of the variable $x$   can have the same value of $\kappa=m_{n-1}$ as $P_n$ has. The result follows.  
\end{proof}

Now we settle Theorem~\ref{th:monot}. 

\begin{proof}  We will use induction. The base of induction is that $4=m_1>m_2=\frac{27}{8}$. Assume now that the statement is proved for $m_{i-1}$ and we want to show that $m_{i-1}>m_i$. By Theorem~\ref{th:main} $m_{i-1}$ is attained as the quotient $a_{i-1}^2/a_{i-2}a_i$ at the polynomial $P_i$ which is a monic polynomial of degree $i$. Moreover up to scaling of $x$ the polynomial $P_i$ is unique in $Pol_{i}$ where this minimum is attained.

 Set  $P_i(x)=\sum _{j=0}^i\ga_jx^j$, $\ga_i=1$. The 
quotient $\ga_1^2/\ga_2\ga_0$ coincides with $m_{i-1}$.  
Given a polynomial $R$ denote by $R^{(k)}$  the result 
of the $k$th truncation  of $R$  from the back, i.e.  the polynomial obtained by removing all terms of $R$ of degree smaller than $k$.  

Consider a perturbation $R(x):=P_i(x)+\varepsilon Q(x)$, 
where $\varepsilon >0$ and 
$Q$ is a monic polynomial of  degree $i+1$.  We choose $Q$ such that for all $k=0,\ldots ,i-2$ the truncation  $Q^{(k)}$ 
has a root at the unique negative double root of ${P_i}^{(k)}$.  
(Notice that ${P_i}^{(i-1)}$ has a single negative real root which we do not have to worry about.)

Setting $Q:=x^{n+1}+\sum _{j=0}^{n-2}\alpha _jx^j$ one can easily see that  the latter condition 
yields a triangular linear system (T) for the undetermined 
the coefficients $\alpha _j$. Hence it has a unique real solution which we denote by $Q^*$.  

For $\varepsilon >0$ small enough all coefficients of the polynomial 
$R^*(x)=P_i(x)+\varepsilon Q^*(x)$  are positive. All roots of 
all polynomials ${R^*}^{(k)}$, $k=0,\ldots ,i-1$, are real, distinct and negative. 

Moreover, for the perturbation $R^*$ the quantity 

$$(\ga_1+\varepsilon \alpha _1)^2/
(\ga_2+\varepsilon \alpha _2)(\ga_0+\varepsilon \alpha _0)$$
 tends to $m_{i-1}$ as $\varepsilon \rightarrow 0$. 
Therefore $m_{i}\leq m_{i-1}$. To prove that the latter inequality is strict we argue as follows. 

The quantity $\ga_1^2/\ga_2\ga_0$ does not change when one performs a linear 
change of the variable $x$. Perform such a change after which the polynomial 
$R^*$ (up to a constant factor) becomes $x^{i+1}+x^i+\cdots$. The linear change 
and the subsequent multiplication by a positive number 
increase the coefficient of $x^{i+1}$ and decrease the coefficients of 
$x^l$ for $l<i$. The latter tend to $0$ as $\varepsilon \rightarrow 0$. 

Now consider $P_{i+1}$. The above means that one can find 
$\delta >0$ and a sequence of degree $i+1$ reverted section hyperbolic polynomials $\{Z\}$ 
remaining outside the ball $B_{\delta}$ 
centered at $P_{i+1}$ and of radius $\delta$ and for 
which the quantity $\ga_1^2/\ga_2\ga_0$ tends to $m_i$. Indeed, all coefficients 
of $P_{i+1}$ are positive while some of the ones of $\{Z\}$ tend to $0$. 

One knows that the minimal value of the quantity $\ga_1^2/\ga_2\ga_0$ in $Pol_{i+1}$  is 
attained only at $P_{i+1}$ up to a scaling. Therefore there exists $\eta >0$ such that 
for all reverted section hyperbolic polynomials from $\partial B_{\delta}$ their quantity $\ga_1^2/\ga_2\ga_0$ 
exceeds  $m_{i}+\eta$. 

On the other hand, similarly to what we did while proving part (iii) of Lemma~\ref{lm:simplex} 
one can define a procedure of continuously changing 
a polynomial $Z$ into the polynomial $P_{i+1}$ so that the quantity $\ga_1^2/\ga_2\ga_0$ 
strictly decreases. The continuous deformation intersects $B_{\delta}$. 
Hence $B_{\delta}$ contains degree $i+1$ truncation hyperbolic polynomials whose quantity $\ga_1^2/\ga_2\ga_0$ is 
at the same time bigger than $m_{i}+\eta$ and less than some number 
arbitrarily close to $m_{i-1}$. Hence $m_{i}<m_{i-1}$. 
\end{proof}

\medskip 
In Lemma~\ref{lm:simplex} we proved that $\De_n^1$ is a curvilinear $(n-1)$-dimensional simplex with vertices $P_1=x^n+x^{n-1}, P_2=x^n+x^{n-1}+x^{n-2}/4,...,P_n$. Next we describe a corollary of Theorem~\ref{th:monot} and Lemma~\ref{lm:simplex}  about the behavior of the function $\kappa=a_{n-1}^2/a_{n-2}a_n$ on $n-1$ edges of this simplex connecting the most important new vertex $P_n$ with already existing vertices $P_1,...,P_{n-1}$. 

\begin{corollary}\label{cor:restr}
For $i=1,...,n-2$ the restriction of $\kappa=a_{n-1}^2/a_{n-2}a_n$ onto the edge $e_{i,n}$ of $\De_n^1$ connecting $P_i$ to $P_n$ is monotone decreasing from the value $m_{n-i-1}$ to $m_{n-1}$. On the remaining edge $e_{n-1,n}$ the function $\kappa$ decreases from $+\infty$ to $m_{n-1}$.   
\end{corollary} 

Notice that formally $\kappa$ is not defined at $P_1,...,P_{n-1}$ so the claim that $\kappa(P_i)=m_{n-i-1}$ in  Corollary~\ref{cor:restr}  should be understood as a limit.  

\medskip 

We now proceed with Theorem~\ref{th:limit} whose proof  requires a number of intermediate steps. 
In fact, we prove a more general statement. Recall that the sequence $\{p_i\},\;i=1,2,...$ was obtained starting from $p_1=1+x$ by adding to the previous polynomial $p_{i-1}$ the maximal possible term $a_{i}x^{i},\; a_i>0$ such that the resulting polynomial is still hyperbolic. This procedure can be equally well started from an arbitrary polynomial $s_1(x)$ of some positive degree $d$ with all negative and simple roots. Thus, we obtain the sequence $\{s_i\},\;i=1,2,...$ with $s_i(x)=s_{i-1}(x)+\tilde a_{d+i-1}x^{d+i-1}$ about which we will prove that $\lim_{i\to \infty}(\tilde a_i)^2/\tilde a_{i-1}\tilde a_{i+1}=1/\q$, see Theorems~\ref{limit} and \ref{limit2}. If we consider the corresponding sequence  $\{S_i\},\; i=1,2,...$ of the reverted polynomials, then one can easily check  that 
\begin{equation}\label{eq:step}
S_j(x)=x(S_{j-1}(x)-S_{j-1}(\xi_{j-1})), 
\end{equation}
 where $\xi_{j-1}$ is a point of a local minimum of $S_{j-1}(x)$ at which it attains the largest value among all its local minima.
 (Notice that, in general, $\xi_{j-1}$ is not unique.  However, the sequence  $\{S_j\}$ is well-defined since by definition the value $S_{j-1}(\xi_{j-1})$ is the same for all possible choices of $\xi_{j-1}$.)  Formula \eqref{eq:step} is well-defined  if $\deg S_{j-1}\ge 2$. In the exceptional case $\deg S_1=1$ we set $S_2=xS_1$. For $j\ge 2$ define $T_j:=S_j/x$. 

\begin{lemma}\label{lm:6} The following facts hold:  
 
 \noindent
{\rm a)} For $j\geq 2$  the polynomials  $T_j$ have all  negative roots. 
 
 \noindent
{\rm  b)} Exactly one of these roots (namely, the one at $\xi_{j-1}$)  is a double root and the rest are simple. 

\noindent
{\rm c)} For $j\geq 3$ the point $\xi_{j-1}$ is the 
rightmost critical point of $S_{j-1}$; the critical values of $S_{j-1}$ 
at all other local minima are smaller than $S_{j-1}(\xi_{j-1})$, i.e. the absolute values of all other local minima are larger than the one  at $\xi_{j-1}$. 
\end{lemma}

 \begin{proof} 
Denote by $0>x_1>\cdots >x_{d}$ the negative roots of $S_{1}$ and denote by $t$ 
any of its critical points, where  $S_{1}$ has 
a local maximum. (Here  we assume $\deg S_1\ge 3$. If  $\deg S_1\le 2$, then for $S_2$ and $S_3$ the above claim can be checked directly and starting with $S_4$ we can use the same argument as  for the case $\deg S_1 \geq 3$.)   
Suppose that $x_{i+1}<t<x_{i}$. Then the polynomial $T_2(x)=S_{2}(x)/x=S_{1}(x)-S_{1}(\xi_{1})$ 
is hyperbolic, with all roots negative. Under our assumptions the root of $T_{2}(x)$  at $\xi_{1}$ has multiplicity $2$. In principle,  the   critical value $S_{1}(\xi_{1})$ might be attained  at more than one minimum of $S_{1}$, in which case 
the polynomial $T_{2}(x)$ has other double root(s) as well. W.l.o.g. we might assume that  $\xi_{1}$ is the rightmost of these local minima, where the critical value is maximal among all local minima. Then on the interval $(\xi_1,0)$, the polynomial $S_{2}$ has a unique  local minimum. We temporarily denote it by  $\xi_{2}$ and show that at this minimum the critical value is maximal among all minima of $S_{2}$ thus justifying our notation. 
Indeed, if $v=S_{2}(\xi_{2})$ one has
$|v|=\max _{x\in (\xi_{1},0)}|x||S_{1}(x)-S_{1}(\xi_{1})|<|\xi_{1}||S_{1}(\xi_{1})|$. 
As $|t|>|\xi_{1}|$, one obtains
$$|S_{2}(t)|=(|S_{1}(t)|+|S_{1}(\xi_{1})|)|t|>|S_{1}(\xi_{1})||\xi_{1}|=|v|.$$ 

On the other hand, 
there exist two roots $f_{\nu+1}<f_{\nu }<0$ of $S_{2}$ such that 
$t\in (f_{\nu+1}, f_{\nu })$. The polynomial $S_{2}$ has a local minimum on 
$(f_{\nu+1}, f_{\nu })$ and for the critical value $r$ 
of $S_{2}$ at this minimum one has $|r|\geq |S_{2}(t)|$ hence $|r|>|v|$. 

Thus $S_{2}$ can attain the largest value among all its minima only at $\xi_{2}$.
This implies that except one double root at $\xi_{2}$, the polynomial 
$T_3(x)=S_{3}/x$ has all its roots distinct and negative. The same argument proves the statement for all $j> 3$. 
\end{proof}   


\begin{remark}\rm{ 
Lemma~\ref{lm:6}  implies that if the initial polynomial $S_1$ has the property that its rightmost minimum has the largest  (i.e. having the smallest  modulus)  critical value  among all its minima, then the same property  holds for the whole sequence $\{S_j\}$ defined above. }
\end{remark}

In what follows we will always use the latter assumption on $S_1$.  Lemmas~\ref{leq1/3}, \ref{signchange} and 
\ref{summarise1}  summarize  further properties of  the polynomials $S_j$ constructed under this assumption.  





\begin{lemma}\label{leq1/3} For each positive integer $k$ 
one has $\xi _k/\xi _{k-1}\leq 1/3$. Moreover,  equality  takes  place only  for $\deg S_1=1$ and $k=2$.
\end{lemma}

\begin{proof}  Assume that $\deg S_1=d\ge 2$ which implies that $\deg S_k=k+d-1$.  
Set as above  $T_{k-1}:=S_{k-1}(x)-S_{k-1}(\xi _{k-1})$. 
The quantity $\xi _k$ satisfies the equality 
$\xi _kT_{k-1}'(\xi _k)+T_{k-1}(\xi _k)=0$. 
Substituting the variable $x$ by $-\xi_{k-1}x$ we may assume w.l.o.g.  that $\xi _{k-1}=-1$. Hence 

\begin{equation}\label{apriori}
-\frac{1}{\xi} _k=\frac{T_{k-1}'(\xi _k)}{T_{k-1}(\xi _k)}=\frac{2}{\xi _k+1}+
\sum _{j=1}^{k+d-4}\frac{1}{\xi _k-\alpha _j}~,
\end{equation}
where $\alpha _j$ are the roots of $T_{k-1}$ smaller than $-1$ (listed 
in the increasing order). 

The equation $-1/\xi _k=2/(\xi _k+1)$ has the unique solution $\xi _k=-1/3$. One can easily check that this is  the solution of (\ref{apriori}) only for $\deg S_1=1$ and $k=2$ in which case $S_1=(x+2)$, $S_2=x(x+2)$ and $S_3=x(x+1)^2$  (up to rescaling).  
 For larger $k$  the presence of 
the  additional summand $\sum _{j=1}^{k+d-4}1/(\xi _k-\alpha _j)$ in the right-hand side implies that 
 the graphs of the l.h.s. and the r.h.s. of \eqref{apriori}  are  intersecting each other closer to the origin than $-1/3$. Indeed, 
the function $-1/\xi _k$ is increasing on $(-1,0)$ while the functions 
$2/(\xi _k+1)$ and $\sum _{j=1}^{k+d-4}1/(\xi _k-\alpha _j)$ are decreasing 
there. Each of these functions takes positive values on $(-1,0)$. 
The functions $-1/\xi _k$ and $2/(\xi _k+1)$ tend to $+\infty$ when their 
arguments tend to $0$ and $-1$ respectively.
\end{proof}   

For a given initial polynomial $S_1$ as above define $A_m=-S_{m-1}(\xi _{m-1})$, i.e. $A_m$ is the absolute value of the largest minimum of $S_{m-1}$. 

\begin{lemma}\label{estimate}
For $l>m>1$ one has  
$A_l\leq A_m(4|\xi _{m-1}|)^{l-m}/3^{(l-m)(l-m+5)/2}$.
\end{lemma}

\begin{proof}
By Taylor's formula  applied at $\xi_{m-1}$ one has for $x\in (\xi _m,0]$ that 
$S_{m-1}(x)-S_{m-1}(\xi _{m-1})=(x-\xi _{m-1})^2b(t_x)$, 
where 
$t_x\in (\xi _{m-1},x)$ and $b=S_{m-1}^{\prime\prime}$. Hence 
$S_m(x)=x(x-\xi _{m-1})^2b(t_x)$. 

The function $b$ is non-decreasing on 
$[\xi _m,0]$ and for $m>2$ it is strictly increasing. 
Indeed, $b(t)$  is the second derivative of the hyperbolic 
polynomial $S_{m-1}$ having all its roots smaller than $\xi_{m-1}<\xi _m$. 
Therefore,  $b(\xi _{m-1})>0$ for  $l=0,1,\ldots , \deg S_{m-1}-1$. 

For $m\geq 2$ the quantity $t_x$ is an increasing function of $x$. Indeed,   
consider the functions $F:=(x-\xi _{m-1})^2b(t_x)$ and 
$G:=(x-\xi _{m-1})^2b(t_{x_1})$. 
For $0\geq x_2>x_1>\xi _{m-1}$ one has $F(x_1)=G(x_1)$ and 
$F(x_2)>G(x_2)$. Therefore $b(t_{x_2})>b(t_{x_1})$ hence $t_{x_2}>t_{x_1}$. 

Consider the quantity $|S_m(x)|=|x(x-\xi _{m-1})^2b(t_x)|$. Its maximum on 
$[\xi _{m-1},0]$ equals $A_{m+1}=|S_{m}(\xi _m)|$. 
Set $R(x):=|x(x-\xi _{m-1})^2|$. Hence  

$$\max _{x\in [\xi _{m-1},0]}R(x)=R\left(\frac{\xi _{m-1}}{3}\right)=\frac{4}{27}|\xi _{m-1}|^3~.$$
On the other hand, 

$$A_{m+1}=|S_{m}(\xi _m)|=R(\xi _m)b(t_{\xi _m})<R\left(\frac{\xi _{m-1}}{3}\right)b(t_0)=$$ 
$$=\frac{4}{27}|\xi _{m-1}|\xi _{m-1}^2b(t_0)=\frac{4}{27}|\xi _{m-1}|A_m.$$
This is the required inequality  for $l=m+1$. To obtain it for $l=m+2$ 
recall that $|\xi _m|\leq |\xi _{m-1}|/3$, 
by Lemma~\ref{leq1/3}. Hence 

$$A_{m+2}<\frac{4}{27}|\xi _m|A_{m+1}\leq \frac{1}{3}\left(\frac{4}{27}\right)^2\xi _{m-1}^2A_m~.$$
Suppose that $A_l\leq A_m(4\xi _m^{l-m})/3^{(l-m)(l-m+5)/2}$. Then  
$|\xi _{l-1}|\leq |\xi _{m-1}|/3^{l-m}$ and 

$$A_{l+1}<\frac{4}{27}|\xi _{l-1}|A_l\leq \frac{4}{27}A_m(|\xi _{m-1}|/3^{l-m})
(4|\xi _{m-1}|)^{l-m}/3^{(l-m)(l-m+5)/2}=$$

$$=A_m(4|\xi _{m-1}|)^{l-m+1}/3^{(l-m+1)(l-m+6)/2}~,$$
which  proves Lemma~\ref{estimate}  by induction on $l$.
\end{proof}

\begin{lemma}\label{signchange}
For $1\leq s\leq k-2$ one has   
${\rm sgn}\; S_k(\xi _s)=(-1)^{k-s+1}$.  
\end{lemma}


\begin{proof}
By definition 
$S_{m+1}(x)=x(S_m(x)+A_{m+1})$. Hence for $l>m$ one gets 

$$S_l(x)=x^{l-m}S_m(x)+\sum _{j=m+1}^lA_jx^{l-j+1}~.$$
As $S_m(\xi _{m-1})=0$, one has 
$S_l(\xi _{m-1})=\sum _{j=m+1}^lA_j\xi _{m-1}^{l-j+1}$. The signs of the 
terms in this sum alternate (because $\xi _{m-1}<0$ and $A_m>0$). By Lemma~\ref{estimate} their 
absolute values  rapidly  decrease and it is the sign of 
$A_{m+1}\xi _{m-1}^{l-m}$ which defines the sign of $S_l(\xi _{m-1})$. 
Indeed, compare this term with the quantity

$$B:=\left|\sum _{j=m+2}^lA_j\xi _{m-1}^{l-j+1}\right|\leq 
\sum _{j=m+2}^lA_j|\xi _{m-1}^{l-j+1}|\leq 
A_{m+1}|\xi _{m-1}^{l-m}|\sum _{j=m+2}^l\frac{4^{j-m-1}}{3^{(j-m-1)(j-m+4)/2}}~.$$
The sum in the right-hand side is majorized by 
$\sum _{\nu =1}^{\infty}4^{\nu}/3^{3\nu}=4/23$, so 

$$B\leq \frac{4}{23}A_{m+1}|\xi _{m-1}^{l-m}|~~,~~|S_l(\xi _{m-1})|\geq 
\frac{19}{23}A_{m+1}|\xi _{m-1}^{l-m}|$$
and ${\rm sgn}\;S_l(\xi _{m-1})={\rm sgn}\; \xi _{m-1}^{l-m}=(-1)^{l-m}$. Setting $s=m-1$ we get the required statement. 
\end{proof}

\begin{lemma}\label{summarise1}
One has $\xi _k/\xi _{k-1}>0.2864887043$.
\end{lemma}

\begin{proof} 
Use the notation in  the proof of Lemma~\ref{leq1/3}. On the interval $(-1,0)$
the right-hand side of equation (\ref{apriori}) is majorized  by 
$2/(\xi _k+1)+\sum _{j=1}^{k+d-4}1/(\xi _k+3^j)$. Indeed, one has 
$\alpha _j\in (\xi _{j+1},\xi _{j+2})$. To majorize  one can replace 
$1/(\xi _k-\alpha _j)$ by $1/(\xi _k-\xi _{j+1})$ and then use part a)  
of Lemma~\ref{lm:6}. 

As $\xi _k>-1$, one can further replace $1/(\xi _k+3^j)$ by 
$1/(3^j-1)$. For $j\geq 2$ one has $3^j-1>2(3^{j-1}-1)$ and further   

$$\sum _{j=1}^{k+d-4}\frac{1}{(3^j-1)} <\sum _{j=1}^{\infty}\frac{1}{(3^j-1)}<\frac{1}{2}+\frac{1}{8}+\frac{1}{26}+\frac{1}{80}+\frac{1}{121}<\frac{11}{16},$$
where $\sum _{j=5}^{\infty}1/(3^j-1)$ is majorized by twice its first term which equals 
$1/242$. 
Therefore $\xi _k$ will be majorized by the solution of the equation 
$$-\frac{1}{\xi _k}=\frac{2}{\xi _k+1}+\frac{11}{16}$$ which belongs to $(-1,0)$. 
The latter equals $-0.2864887043...$ implying that  $\xi _k/\xi _{k-1}>0.2864887043$.
\end{proof}  

\begin{remark} \rm{
The above upper and lower  bounds for  $\xi _k/\xi _{k-1}$ are quite close to one another and 
imply that the quantities $|\xi _k|$ decrease approximately  
as a falling geometric progression. }Ê 
\end{remark}

Define $\Phi (x):=-{1}/{x}-{2}/{(x+1)}$,  $\psi (r):=\sum _{j=1}^{\infty}{r^j}/{(1-r^{j+1})}$. The next result is central in the proof of Theorem~\ref{th:limit}. 

\begin{theorem}\label{limit}
{\rm (i)} The limit $\la=\lim_{k\to\infty}\xi _k/\xi _{k-1}$ exists;  

\noindent
{\rm (ii)} $\la$  is the unique solution of the equation 
\begin{equation}\label{eq:main}
\Phi (-\lambda )=\psi (\lambda),
\end{equation} 
belonging to $(0,1).$
\end{theorem}

\begin{proof} 
Set $l_0:=0.2864887043...$ and $r_0=1/3$. 
Lemmas~\ref{leq1/3} and   ~\ref{summarise1} imply that if $\lambda$ exists, 
then it belongs to $I_0:=[l_0,r_0]$. As in the proof of 
Lemma~\ref{leq1/3} rescale  the variable 
$x$ to obtain  $\xi _{k-1}=-1$. 

We construct a series of closed intervals $I_i:=[l_i,r_i]$, where 
$l_i<l_{i+1}<r_{i+1}<r_i$, such that for each fixed $i$ one has 
$\xi _{k-1}\in I_i$ for $k$ sufficiently large. 

Consider equation \eqref{apriori} and set 
$U:=\sum _{j=1}^{k+d-4}1/(\xi _k-\alpha _j)$. Recall that 
$\alpha _j \in (\xi _{j+1},\xi _{j+2})$ for $j\geq 1$ where $\al_j$  are the roots of $T_{k-1}$ in the increasing order.  
We can decrease the value of $U$  at every  point of the interval $(-1,0)$ 
by assuming that for these $j$ one has  $\alpha _j=\xi _{j+1}$, and then 
by requiring $\xi _{j+1}$ to be as small as  possible. The last condition 
means that each ratio $\xi _m/\xi _{m-1}$ equals $l_0$ for all  $m=1,\ldots ,k-2$.  

For each  fixed $k$ the solution to (the modified as above) equation 
(\ref{apriori}) will exceed $-1/3$ 
because $-1/3$ was obtained in the absence of the sum $U$. These 
solutions increase with $k$  because when $k$ increases  more and more terms are added to $U$. 
Denote by $-r_1$ the limit of these solutions as $k\rightarrow \infty$. 
We get $r_1<r_0=1/3$. 

It is clear that if the limit $\la$  exists, then it must 
belong to the interval $[l_0,r_1]$. Moreover, all accumulation points 
 of the sequence of solutions  to the original equation (\ref{apriori}) when $k\rightarrow \infty$ belong to $[l_0,r_1]$.

Analogously, the value of $U$ increases  on the whole interval $(-1,0)$  
if one sets $\alpha _j=\xi _{j+2}$ ($j\geq 1$) and  requires 
$\xi _{j+2}$ to be the maximal possible. In this case each ratio 
$\xi _m/\xi _{m-1}$ equals $r_1$ for all $m=1,\ldots ,k-2$.  

Hence for each fixed $k$  the solution to (the modified as above) equation 
(\ref{apriori}) will be smaller than $-l_0$ because $-l_0$ was obtained when
 these ratios were equal to $r_0>r_1$. Denote by $-l_1$ the limit 
of these solutions as $k\rightarrow \infty$. As above we get  that 
$l_1>l_0$. Thus $l_0<l_1<r_1<r_0$. The construction of all further quantities 
$l_i$ and $r_i$  follows the same pattern.

Let us  show that when   $i\rightarrow \infty$ the lengths of the intervals $I_i$ tend to $0$  as fast as a falling geometric progression implying that 
their common intersection is a single point. In other words, there is only one accumulation 
point of the sequence of solutions to the initial equation (\ref{apriori}) 
as $k\rightarrow \infty$. 

To do this  present  \eqref{apriori} in the form $\Phi (x)=U(x)$ 
and let $k\rightarrow \infty$. 
The modified equations will have  the form

\begin{equation}\label{apriori1}
\Phi (x)=\varphi (r_i,x)\quad \text{and} \quad \Phi (x)=\varphi (l_i,x),
\end{equation}
where $\varphi (r,x):=\sum _{j=1}^{\infty }1/(x+(1/r)^j)$. 
The left (respectively right) equation in (\ref{apriori1})
has $-l_{i+1}$ (respectively  $-r_{i+1}$) as its solution on $(0,1)$. 
The series $\varphi$ converges uniformly on $[l_0,r_0]\times [-1,0]$. 

Each of the functions $\varphi (r_i,x)$ and $\varphi (l_i,x)$ 
is decreasing on $(0,1)$. For each fixed $x\in (0,1)$  one has 
$\varphi (r_i,x)>\varphi (l_i,x)$. Therefore, the intersection points 
of the graph 
of $\Phi (x)$ with that of $\varphi (r_i,x)$ and $\varphi (l_i,x)$ 
belong to the rectangle 
$[-r_i,-l_i]\times [\varphi (l_i,-l_i), \varphi (r_i,-r_i)]$. 

For $x\in [-r_0,-l_0]$ one has $1/x^2\geq 9$ and $2/(x+1)^2>\frac{32}{9}>3$, hence 
$|\Phi '(x)|>12$. Therefore 

$$|r_{i+1}-l_{i+1}|<\frac{1}{12}|\varphi (r_i,-r_i)-\varphi (l_i,-l_i)|~.$$

To simplify  the notation we write $r$  instead of $r_i$ and $l$ instead of  $l_i$. Set 

$$M:=\varphi (r,-r)-\varphi (l,-l)=
\sum _{j=1}^{\infty}(r^j-l^j)/((1-r^{j+1})(1-l^{j+1}))~.$$
For each $j$ there exists $\theta _j\in (l,r)$ such that 
$r^j-l^j=j\theta _j^{j-1}(r-l)<jr^{j-1}(r-l)$. As $l\leq r\leq 1/3$, one has 
$(1-r^{j+1})(1-l^{j+1})>(2/3)^2=4/9$. Thus $0\leq M\leq ({9}/{4})jr^{j-1}(r-l)$.
Recall that $\sum _{j=1}^{\infty}jr^{j-1}=1/(1-r)^2\leq 9/4$. Therefore, 

$$|r_{i+1}-l_{i+1}|<|\varphi (r_i,-r_i)-\varphi (l_i,-l_i)|\leq 
\frac{81}{12\cdot 16}(r_i-l_i)<\frac{r_i-l_i}{2}~.$$
This proves part (i) of the theorem.

To settle part (ii) one has to observe that the solution 
$-l_{i+1}$ to equation (\ref{apriori1}) and the parameter $-r_i$ both tend 
to $-\lambda$, while 

$$\varphi (r,-r)=\sum _{j=1}^{\infty }\frac{r^j}{1-r^{j+1}} 
=\psi (r)~.$$ 
Therefore, $\lambda$ solves equation~\eqref{eq:main}.
\end{proof}   

\begin{remark}\label{rm:V}\rm{
The number $\lambda$ is defined as limit when  $k\rightarrow \infty$ 
of the critical points of the function $xT_k(-\xi_kx)$ belonging to $(0,1)$.
Equation (\ref{apriori}) has a solution in every interval 
$(\alpha _j,\alpha _{j+1})$. For every fixed $j-k$  there exists the limit 
as $k\rightarrow \infty$ of the solution belonging to 
$(\alpha _{j-k},\alpha _{j-k+1})$. The proof is just the same as the one 
in the above theorem. Denote all these solutions by 
$\zeta _1>\zeta _2>\cdots$. 

The quantities $|\alpha _j|$ are growing as a geometric progression. 
Therefore the infinite product 
$$W:=(x+1)\prod _{j=1}^{\infty}\left(1-\frac{x}{\zeta _j}\right)$$ 
is an entire function of genus $0$. Hence for every fixed $j-k$  there exists 
$\widetilde \al_j=\lim_ {k\rightarrow \infty} \alpha_{j-k}$. Define the entire
function $V$ by the relation  
\begin{equation} \label{eq:V}
V(x):=\int _{-1}^xW(t)dt\left/\int _{-1}^0W(t)dt\right.
\end{equation}
Then the above quantities $\widetilde \al_j$  are the zeros of the function $V$. 
Notice that with the above normalization one gets $V(0)=1$.} 
\end{remark}

\begin{theorem}\label{limit2}
The quantity $A_m^2/(A_{m-1}A_{m+1})$ 
tends to $1/\lambda$ as $m\rightarrow \infty$.
\end{theorem}

\begin{proof} 
By Taylor's formula of order $2$ applied 
at $\xi _{m-1}$ the polynomial $T_{m-1}$ has the form 
$T(x)=(x-\xi _{m-1})^2S_{m-1}''(\xi _{m-1}+
\theta (x-\xi _{m-1})(x-\xi _{m-1}))$, 
where $\theta (x-\xi _{m-1})\in (0,1)$. Notice that for $x>\xi _{m-1}$, 
$\theta (x-\xi _{m-1})$ is uniquely defined. Indeed, $S_{m-1}''$ is 
increasing for $x>\xi _{m-1}$ as the second derivative of 
a hyperbolic polynomial with all roots smaller than $\xi _{m-1}$.  

Thus $A_{m-1}=T_{m-1}(0)=\xi _{m-1}^2S_{m-1}''(\xi _{m-1}+
\theta (-\xi _{m-1})(-\xi _{m-1}))$. 
Notice that  $A_m=\max _{[\xi _{m-1},0]}|x(x-\xi _{m-1})^2S_{m-1}''(\xi _{m-1}+
\theta (x-\xi _{m-1})(x-\xi _{m-1}))|$. The maximum is attained at $\xi _m$ and $\xi _m/\xi _{m-1}\rightarrow \lambda$ as 
$m\rightarrow \infty$. Thus 

$$\frac{A_m}{A_{m-1}}=\xi _m\frac{(\xi _m-\xi _{m-1})^2}{\xi _{m-1}^2}
\frac{S_{m-1}''(\xi _{m-1}+
\theta (\xi _m-\xi _{m-1})(\xi _m-\xi _{m-1}))}{S_{m-1}''(\xi _{m-1}+
\theta (-\xi _{m-1})(-\xi _{m-1}))}~.$$
The first quotient to the right (denoted by $Y_m$)  
tends to $(\lambda -1)^2$. Denote by $F_m$ the 
second quotient. In the same way one shows that 
$A_{m+1}/A_m=\xi _{m+1}Y_{m+1}F_{m+1}$. Hence 

$$\frac{A_m^2}{A_{m-1}A_{m+1}}=\frac{\xi _m}{\xi _{m+1}}\frac{Y_m}{Y_{m+1}}\frac{F_m}{F_{m+1}},$$ 
with $\lim _{m\rightarrow \infty}(\xi _m/\xi _{m+1})=1/\lambda$ and 
$\lim _{m\rightarrow \infty}(Y_m/Y_{m+1})=(\lambda -1)^2/(\lambda -1)^2=1$. 
So to prove the theorem we need  to show that $\lim _{m\rightarrow \infty}(F_m/F_{m+1})=1$.

The latter statement follows from the above remark. Indeed, set 
$x\mapsto |\xi _{m-1}|x$ (resp. $x\mapsto |\xi _{m}|x$). Then the 
function $T_{m-1}$ (resp. $T_m$), in the limit as $m\rightarrow \infty$, 
becomes the function $V$ defined by \eqref{eq:V}. 
The numerators of the quotients $F_m$ and $F_{m+1}$ tend to 
$V''(-1+\theta (-\lambda +1)(-\lambda +1))$ while the denominators 
tend to $V''(-1+\theta (1))$. Hence $F_m/F_{m+1}\rightarrow 1$ as $m\rightarrow
\infty$.
\end{proof} 

\begin{proposition}\label{pr:V}
The function $V$ defined by \eqref{eq:V} enjoys the following properties:
\begin{itemize}
\item[\rm (i)] $V(0)=1,\; V(-1)=V'(-1)=0$;\\
\item[\rm(ii)] $V$ belongs to the Laguerre-P\'olya class $\LP^+$;\\
\item[\rm (iii)] $V$ satisfies the functional relation:  $V(x)=1+{xV(\lambda x)}/{V(-\lambda )}.$
\end{itemize}
\end{proposition}

Notice that the latter relation implies that for any choice of $\lambda$ one has $V(0)=1$ and $V(-1)=0$. On the other hand, $V'(-1)=0$ is an additional condition which together with ${\rm(ii)}$ determines $\lambda$. 

\begin{proof}[Proof of Proposition~\ref{pr:V}]
Part (i) follows from the definition of the function $V$, see Remark~\ref{rm:V}. 
To prove part (ii) notice that  
the functions $V$ and $W$ are limits of sequences of hyperbolic
polynomials $H_k$ with negative roots. 
If the roots are numbered in the order of 
increasing  absolute values, 
then for every fixed $j$ the root $\alpha _j$ 
has a finite limit when $k\to \infty$. 

The modules of the roots increase faster than a geometric progression 
with ratio $2.6$. 
Hence the sequence of polynomials is uniformly convergent on any compact set 
$\Omega$. 
Indeed, consider the product $\prod _{i=N}^{\infty}(1-x/\alpha _i)$. The 
module of its logarithm is majorized by $C\sum _{i=N}^{\infty}|x/\alpha _i|$ 
(where $C>0$ depends only on the set $\Omega$) which 
is arbitrarily and uniformly on $\Omega$ small  
if one chooses $N$ sufficiently large. 
To see this  notice that $|\ln (1+y)|<|y|+|y|^2+|y|^3+\cdots$, 
where $y=-x/\alpha _i$, for $i$ large enough the sum 
(which equals $|y|/(1-|y|)$) is smaller than $2|y|$.  
The limit of such a sequence belongs to the class $\LP^+$ 
by definition.  

To prove part (iii) one has to recall the equality 
$S_{m+1}(x)=x(S_m(x)-S_m(\xi _m))$ or, equivalently, 
$T_{m+1}(x)=xT_m(x)-\xi _mT_m(\xi _m)$. The function $V$ is the limit when 
$m\rightarrow \infty$ of the polynomials $T_m(-\xi_{m-1}x)$  multiplied by a nonzero constant so that $T_m(0)=1$.  
Rescaling sends the double root of $T_m(x)$  to  $-1$. 
\end{proof}

In order to finish the proof of Theorem~\ref{th:main} we study the problem which entire functions satisfy the properties given in  Proposition~\ref{pr:V}. The following definition is crucial for our further considerations. Recall that $\Psi(q,u)=\sum_{j=0}^\infty q^{\binom{j+1}{2}}u^j.$  
We say that a pair $(\hat q,\hat u)$ is {\em critical} for $\Psi(q,u)$ if $|\hat q|<1$ and $\Psi(\hat q,u)$ as a function of $u$ has a double root at $\hat u$.  

\begin{theorem}\label{th:relation} There exists  an analytic in a disk $|x|\le r,\; r>1$ function $V(x)$ satisfying the relation 
\begin{equation}\label{eq:rel}
V(x)=-\hat u\hat q(xV(\hat qx)+V(-\hat q))
\end{equation} 
for some $\hat u\in \bC^*$ and $|\hat q|<1$ as well as 
the  boundary condition $V'(-1)=0$ 
if and only if the pair $(\hat q, \hat u)$ is critical. 
\end{theorem} 

Notice that \eqref{eq:rel} is exactly the relation {\rm (iii)} of Proposition~\ref{pr:V} with an undefined scalar factor $\hat u$.

\begin{proof}[Proof of Theorem~\ref{th:relation}] 
Assume that a function $V(x)$  analytic in some disk $|x|\le r,\; r>1$ satisfies the relation $V(x)=-uq(xV(qx)+V(-q))$ for some fixed $u\neq 0$ and $q\neq 0$. W.l.o.g. we can assume $V(0)=1$ which is equivalent to $\be_0=1$. Substituting the power series $V(x)=\sum_{j=0}^\infty \be_jx^j$ in the latter relation one gets 
$$\sum_{j=0}^\infty \be_jx^j=-quV(-q)-qux\sum_{j=0}^\infty \be_jq^jx^j.$$
Comparing the coefficients at equal 
powers in the latter relation  we get the system of equalities
\begin{equation}\label{eq:1}
-quV(-q)=1, \quad \text{and} 
\end{equation}

$$ \be_1=-qu,\;\be_2=-q^2u\be_1,\; \be_3=-q^3u\be_2,\;... \;, \be_k=-q^{k}u\be_{k-1}, ...$$
implying that $\be_j=q^{\binom {j+1}{2}}(-u)^{j},\; j=1, 2,3,...$. Substituting these coefficients in $V(x)$ we get 
 $V(x)=\sum_{j=0}^\infty q^{\binom{j+1}{2}}(-ux)^j.$  In terms  of the partial theta  function $\Psi(q,u)$ given by \eqref{eq:Psi} one gets  $V(x)=\Psi(q,-ux)$ and condition \eqref{eq:1} takes the form  $\Psi(q,qu)=-\frac{1}{qu}$.  Let us show that  it is  equivalent to $\Psi(q,u)=0$. Indeed, expanding $\Psi(q,qu)=-\frac{1}{qu}$ we get 
 $qu\sum_{j=0}^\infty q^{\binom {j+1}{2}}(qu)^j=-1 \Leftrightarrow 1+ \sum_{j=0}q^{\binom {j+1}{2}}(qu)^{j+1}=0 \Leftrightarrow 1+ \sum_{j=0}q^{\binom {j+2}{2}}u^{j+1}=0 \Leftrightarrow \Psi(q,u)=0.$  Notice that  for the power series expressing $V(x)$ to have a  convergence radius exceeding $1$ it is necessary and sufficient to have $|q|<1$ in which case $V(x)$ is entire. Now we use the last boundary condition $V'(-1)=0$. With $V(x)=\Psi(q,-ux)$ we get $V'(x)=-u\partial\Psi(q,-ux)$, where $\partial$ stands for the partial derivative w.r.t. second argument. Finally, $V'(-1)=-u\Psi'_u(q,u)$.  Thus one gets the  system 
 $$\begin{cases}\Psi(q,u)=0\\ -u\Psi'_u(q,u)=0\end{cases}  \Leftrightarrow  \begin{cases}\Psi(q,u)=0\\ \Psi'_u(q,u)=0\end{cases},$$
 since $u=0$ is never a solution of $\Psi(q,u)=0$. 
 
An arbitrary solution $(\hat q,\hat u)$ of  the latter system is exactly a critical pair in the above definition, i.e. $\hat q$ is such that the function $\Psi(\hat q, u)$ as a function of $u$ has a double root at $\hat u$. 
\end{proof}

\medskip

Denote by $g_q(x) =\sum_{k=0}^\infty q^{k^2x^k},\; 0 < q < 1$ 
the partial theta-function (in this form), and by
 $S_n(q,x) =\sum_{k=0}^nq^{k^2}x^k$ 
its $n$th Taylor section. We will need the following lemma.

\begin{lemma}\label{lm:finite} For every real $q\in (0, 1)$ there exists a number $m\in \bN$ such that for
all $n\ge 2m+2$ the number of non-real zeros of $S_n(q,x)$ is not greater than $2m+2$.
\end{lemma} 

\begin{proof}  We will use the following well-known identity for $|q|<1$ 
$$\prod_{k=1}^\infty\frac{1-q^{2k}}{1 + q^k} = 1 + 2 \sum_{k=1}^\infty(-1)^kq^{k^2}$$
(see e.g. \cite{PSz}, Chapter 1, Problem 56). By this identity we have
$$1 + 2\sum_{k=1}^\infty(-1)^kq^{k^2}> 0,\;  q \in (0, 1).$$ 
Thus for every $q \in (0, 1)$ there exists $m = 2s+1 \in  \bN$ such that the inequality holds
\begin{equation}\label{eq:hjlp}
\quad 1 + 2\sum_{k=1}^m(-1)^kq^{k^2}> 0.
\end{equation}
Then for every $n\ge 2m + 2$ we have for all $k,\; m + 1 \le k \le n - m - 1$:
$$(-1)^kS_n(q,-\frac{1}{q^{2k}}) =\sum_{j=0}^n(-1)^{j-k}q^{j^2-2kj} = q^{-k^2}\sum_{j=0}^n
(-1)^{j-k}q^{(j-k)2}=$$
$$=q^{-k^2}\left(\sum_{j=0}^{k-m-1}(-1)^{j-k}q^{(j-k)^2}+
\sum_{j=k-m}^{k+m}(-1)^{j-k}q^{(j-k)^2}+\sum_{j=k+m+1}^n(-1)^{j-k}q^{(j-k)^2}\right)=$$
$$=: \Sigma_1 + \Sigma_2 + \Sigma_3.$$
We see that the summands in $\Sigma_1$ are alternating in signs and their moduli
are increasing. So the sign of $\Sigma_1$ coincides with the sign of the $(k - m - 1)$-
th summand. Therefore,  $\rm{sign}\,\Sigma_1 = (-1)^{k-m-1-k} = (-1)^{-m-1} = (-1)^{-2s-2} = 1$. Analogously the summands in 
$\Sigma_3$ are alternating in signs and their moduli are decreasing. So the sign of $\Sigma_3$ coincides with the sign of the $(k + m + 1)$-th summand. In other words, 
$\rm{sign}\, \Sigma_3 = (-1)^{k+m+1-k} = (-1)^{m+1} = (-1)^{2s+2} = 1$. Therefore $\Sigma_1 \ge 0$ and $\Sigma_3 \ge 0.$ 
By \eqref{eq:hjlp} we have
$$\Sigma_2 = q^{-k^2}\sum_{j=k-m}^{k+m}(-1)^{j-k}q^{(j-k)^2}= q^{-k^2}\left(1 + 2\sum_{k=1}^m(-1)^kq^{k^2}\right)> 0.$$
Therefore for every $k, m+1 \le  k \le  n-m-1$, we have $(-1)^kS_n(q,-\frac{ 1}{q^{2k}}) > 0$. Thus
for $n \ge 2m+2$ the polynomial $S_n(q,x)$ has not less than $n-2m-2$ real zeros (and
the number of non-real zeros of $S_n(q,x)$ is not greater than $2m + 2$).
\end{proof}

\begin{corollary}  For every real $q\in (0,1)$  the functions $g_q(x)$ and $\Psi(q,x)$  have a finite number of
non-real zeros. Moreover, the number of non-real zeros is a non-decreasing function of $q$.
\end{corollary}

\begin{proof} Since $\Psi(q,x)$ is obtained from $g_q(x)$ by rescaling of $x$ it suffices to consider $g_q(x)$ only. To prove the first statement fix an arbitrary  $q\in(0,1)$. By Lemma~\ref{lm:finite} we know that there exists $m$ such that 
the number of non-real roots of any Taylor section $S_n(q,x)$ for sufficiently large $n$ does not exceed $m$. Assume that the function $g_q(x)$ has  $l > m$ non-real roots. Take small circles surrounding these roots and not intersecting the real axis. By the Hurwitz theorem all Taylor sections with large $n$ should have exactly  $l$ roots in the union of disks  bounded by these $l$ circles. Contradiction.  

To prove the second statementÊ consider the sequence of Taylor sections. We  prove  that the number of real roots of any section $S_n(q,x)$  and $g_q(x)$ itself is a monotone non-increasing function of $q\in (0,1)$. Indeed, the sequence $q^{n^2}$ is a complex zero decreasing sequence (CZDS) for any $q\in (0,1)$, see e.g. \cite{CC}. Thus for any $0<q_1<q_2<1$ one has that to obtain the section with the value of parameter $q_1$ from that of $q_2$ one has to multiply the coefficients of the former by $(q_1/q_2)^{n^2}$. Since the latter sequence is CZDS the result for sections follows. For $g_q(x)$  the same argument applies since it has only finitely many non-real zeros and is of genus $0$. 
\end{proof}

Finally to finish the proof of Theorem~\ref{th:main} notice that  by Proposition~\ref{pr:V} and Theorem~\ref{th:relation} 
the function $V(x)$ satisfies the functional relation {\rm(iii)}, $V'(-1)=0$ and belongs to $\LP^+$ which implies that  $\la=\q$ and the function $V$ equals $\Psi(\q,-\widetilde ux)$ where $\widetilde u$ is the double root of the function $\Psi(\q,x)$. \qed

\begin{remark}
{\rm A somewhat mysterious equation~\eqref{eq:main} describes the set of all critical points of the limiting function $\Psi(\q, -\tilde u x)$ of which $\tilde u$ is the only critical point belonging to the interval $(0,1)$.Ê }
\end{remark}

\medskip
Notice that  $\Psi(q,x)$ belongs to $\LP^+$ if and only if 
$q\in (0,\q]$ where $\q$ is the constant appearing in Theorem~\ref{th:limit}, comp. Theorem 4 of \cite{KLV1}. 
The following additional statements are very plausible but since we do not need them we have not put enough effort in proving them. 

\begin{conjecture}\label{conj:extras} 

{\rm  (i)} Enumerating the negative roots of  $\Psi(q,u)$ as $-r_1(q)<-r_2(q)<-r_3(q)<... $ in the order of increasing absolute value we get $\lim_{N\to \infty}\frac{r_{N+1}(q)}{r_N(q)}=1/q.$ 

\noindent
{\rm  (ii)} Analogously, if $-r_1'<-r_2'<-r_3'<...$ stands for the real negative roots of $\Psi^\prime_u(q,u)$ or, in other words, for the  real  critical points of $\Psi(q,u)$ we get  $\lim_{N\to \infty}\frac{r'_{N+1}}{r'_N}=1/q.$

\noindent
{\rm (iii)} Finally, let $c_1,c_2,..., $ be the sequence of the critical values of $\Psi(q,u)$ attained at the real critical points $r_1',r_2',...$, then $\lim_{N\to \infty}\frac{c_N^2}{c_{N-1}c_{N+1}}=q.$ 

\end{conjecture}

We call    the set of all $\hat q_i,\; |\hat q_i|<1$ such that the entire function $\Psi(\hat q, u)$ has a double root (denoted by $\hat u_i$), i.e. the pair  $(\hat q_i, \hat u_i)$Ê is critical {\em spectrum}  $\fS$ of partial theta function $\Psi(q,u)$.  We propose the following conjecture about the properties of $\Psi(q,u)$ for $q\in (0,1)$.

\medskip 

\begin{conjecture}\label{conj:infinit}  The spectrum $\fS$ contains  infinitely many positive values $0<\q=\hat q_1<\hat q_2<...<\hat q_N<... <1$, i.e.  there exists an infinite sequence $\{\hat q_i\},\;i=1,2,...$ of numbers in $(0,1)$ such that  for each positive integer $i$ the function $\Psi(\hat q_i,u)$ has a negative double root in the variable $u$.  
\end{conjecture}

Some numerical evidence for the validity of  Conjecture~\ref{conj:infinit} was newly obtained by two talented high school students A.~Broms and I.~Nilsson. Namely, they calculated the first 25 values of $\hat q_i$ with 12 decimal places. Their list with $6$ decimals is as follows:  
$0.309249, 0.516959, 0.630628,$ $ 0.701265, 0.749269, 0.783984, 
0.810251, 0.830816,$ $ 0.847353,$ $ 0.860942, 0.872305, 0.881949,$ $ 0.890237, 
0.897435, 0.903747, 0.909325,$ $0.914291,$ $ 0.918741,$ $ 0.922751, 0.926384, 0.929689, 0.932711, 0.935482, 0.938035, 0.940393$, see Fig.~\ref{fig58}. 

\medskip

\begin{problem} {\rm Assuming the Conjecture~\ref{conj:infinit} holds what is the asymptotics of the sequence $\{\hat q_i\}$ when $i\to \infty$ ? }
\end{problem}

Let us also mention  the following question posed  by Professor A.~Sokal.  

\begin{problem}
{\rm Is it true that $\fS$ is empty within the open disk $|q|<\q$ ?}
\end{problem}




\medskip 

We finally  prove the remaining Proposition~\ref{prop:log. convexity}. 
Denote by $Pol_n^+\subset Pol_n$ the  set  of all monic degree $n$ polynomials  with all  positive 
coefficients. It contains $\Sigma_n$, the set of degree $n$ polynomials 
with all roots real negative. Let $p(x)=a_nx^n+\cdots +a_0$ be a polynomial.  

\begin{lemma}\label{lm:Hutch}
For any $k=1, 2, \ldots , n-1$ there exists a polynomial $p\in Pol_n^+$, 
$p\not\in \Sigma_n$, for which one 
has $a_i^2\geq 4a_{i-1}a_{i+1}$ for $i\neq k$ and $a_k^2<4a_{k-1}a_{k+1}$.
\end{lemma}

\begin{proof} 
Fix the triple of coefficients $(a_{k-1},a_k,a_{k+1})$ such that 
$a_k^2<4a_{k-1}a_{k+1}$. Hence the polynomial 
$g:=x^{k-1}(a_{k+1}x^2+a_kx+a_{k-1})$ has a $(k-1)$-fold root at $0$ and a 
complex conjugate couple. 

For $i>k+1$ (resp. for $i<k-1$) 
set $a_i=b_i\varepsilon ^{i-k-1}$ (resp. $a_i=b_i\varepsilon ^{k-1-i}$), 
where $\varepsilon \in (0,1]$. 
Choose the coefficients $b_i>0$ such that the inequalities 
$a_i^2\geq 4a_{i-1}a_{i+1}$ hold for $i\neq k$ and $\varepsilon =1$. 
One can do this consecutively.  
E.g. one chooses $b_{k+2}>0$ sufficiently small so that 
$a_{k+1}^2\geq 4a_kb_{k+2}$, then 
$b_{k+3}$ such that $b_{k+2}^2\geq 4b_{k+3}a_{k+1}$ etc. Then in the same way 
$b_{k-2}$, $b_{k-3}$ etc. 

If the inequalities 
$a_i^2\geq 4a_{i-1}a_{i+1}$ ($i\neq k$) hold for $\varepsilon =1$, 
then they hold for any $\varepsilon \in (0,1]$ (to be checked directly). 

For $\varepsilon$ small enough the polynomial $p$ is a perturbation of 
the polynomial $g$. Fix two circles centered at the complex 
roots of $g$ and not intersecting the real axis. 
For $\varepsilon >0$ small enough one has $|p-g|<|g|$ on these circles. 
By the Hurwitz theorem $p$ has 
a root inside each of them. Hence $p\not\in \Sigma_n$.
\end{proof} 

\begin{figure}

\begin{center}
\includegraphics[scale=0.45]{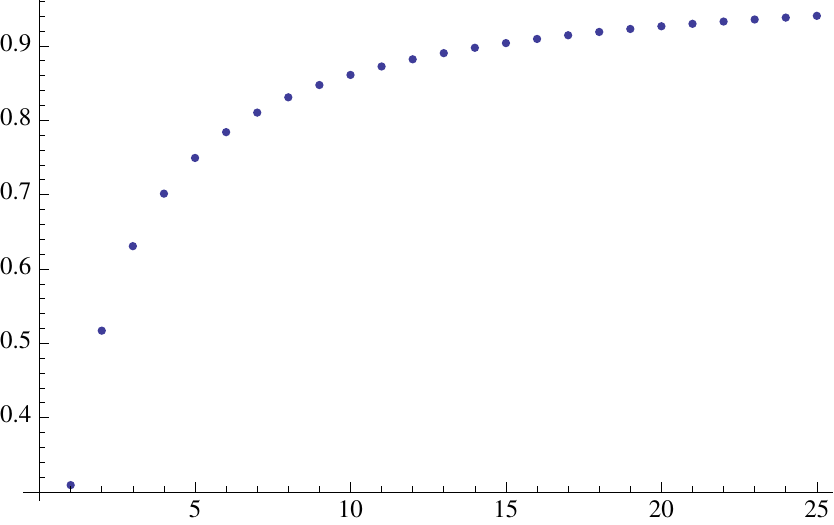}
\end{center}
\vskip 0.3cm

\caption {The first 25 values of $\hat q_i$.}
\label{fig58}
\end{figure}

\begin{proof}[Proof of Proposition~\ref{prop:log. convexity}]
To prove {\rm(i)} notice that by Lemma~5 of \cite{PRS}  Hutchinson's cone is the minimal polyhedral cone containing the set of the so-called sign-independently hyperbolic polynomials and is, on the other hand, contained in $\Delta_n\subset \Sigma_n$.  (A sign-independently hyperbolic polynomial is a hyperbolic polynomial with all positive coefficients and such that any sign change of its coefficient results into a hyperbolic polynomial, see \cite{PRS}.) 
  By Lemma~\ref{lm:Hutch} we see that an arbitrarily small parallel translation 'outward' of any of the hyperplanes defining the logarithmic image of Hutchinson's cone  results into getting outside the logarithmic image of the largest set $\Sigma_n$. Therefore, 
the logarithmic image of Hutchinson's cone is the largest polyhedral cone contained in $L\De_n$ and, analogously, in $L\Sigma_n$.   
 
 To prove {\rm(ii)} observe that Theorem~\ref{th:main} can be interpreted as follows. Consider Hutchinson's cone in the space $Pol_n^1$ and its image under the logarithmic map. Then it has a unique apex, i.e. the vertex where all inequalities become equalities.  Look for a parallel translation of the logarithmic image of Hutchinson's cone  containing the whole $L\De_n$. Then if you take the parallel translation when this apex is placed at the logarithmic image of $P_n$, then the whole $L\De_n$ is covered. Since the translated cone and $L\De_n$ still have a common point this position is minimal for containment. Exactly the same argument  using the known properties of Newton's inequalities tells us that placing the apex at the logarithmic image of $(x+1/n)^n$ does the job.   
\end{proof}

\section {An interesting iteration scheme}

\medskip



The main ingredient  in the proof of  Theorem~\ref{th:limit} is  construction of the sequence $\{S_j\}$ starting from a hyperbolic polynomial $S_1$ with all simple roots the rightmost of which is at the origin. The next polynomial is obtained from the previous one by subtracting   its least in absolute value minimum followed by multiplication by $x$.   We have shown that after appropriate scaling the limiting entire function is a specialization of a partial theta function and also that the critical point at which the minimum is located asymptotically stabilizes.  These results naturally lead to the question about what happens if we consider a similar iterative procedure where the point at which we take a value to subtract is fixed from the beginning. A more detailed consideration leads to the following natural set-up. 
 
Given an initial analytic function $f_1(x)$ defined (at least) in a small open neighborhood of the interval $[-1,0]$ on the real line and a number $0<q<1$ define the sequence $\{f_j\}$ given by
\begin{equation}\label{eq:rec}
f_{j}(x)=\left(1+\frac{xf_{j-1}(qx)}{f_{j-1}(-q)}\right),\; j=2,3,....
\end{equation}
Obviously, $f_{j}(x)$ will be well-defined and analytic in the same neighborhood of $[-1,0]$ unless $-q$ is a root of $f_{j-1}(x)$. For generic choices of $f_1$ the latter circumstance never happens. One can easily check that  for a positive integer $j$  if $f_j$ is well-defined,  then it satisfies the normalization conditions $f_{j}(0)=1$ and $f_{j}(-1)=0$. 

Looking at the first  part of the proof of Theorem~\ref{th:relation} before we use the additional condition $V'(-1)=0$ one can see that  
the fixed points of \eqref{eq:rec}, i.e. the analytic functions satisfying on $[0,1]$ the functional relation 
$$F(x)=\left(1+\frac{xF(qx)}{F(-q)}\right),$$
with $q\in (0,1)$  fixed are exactly of the form $F(x)=\Psi(q,-\hat u x)$ where $\hat u$ is one of the real roots of the equation 
$\Psi(q,u)=0$. Here  $u$ is the variable. 

\begin{figure}

\begin{center}
\includegraphics[scale=0.35]{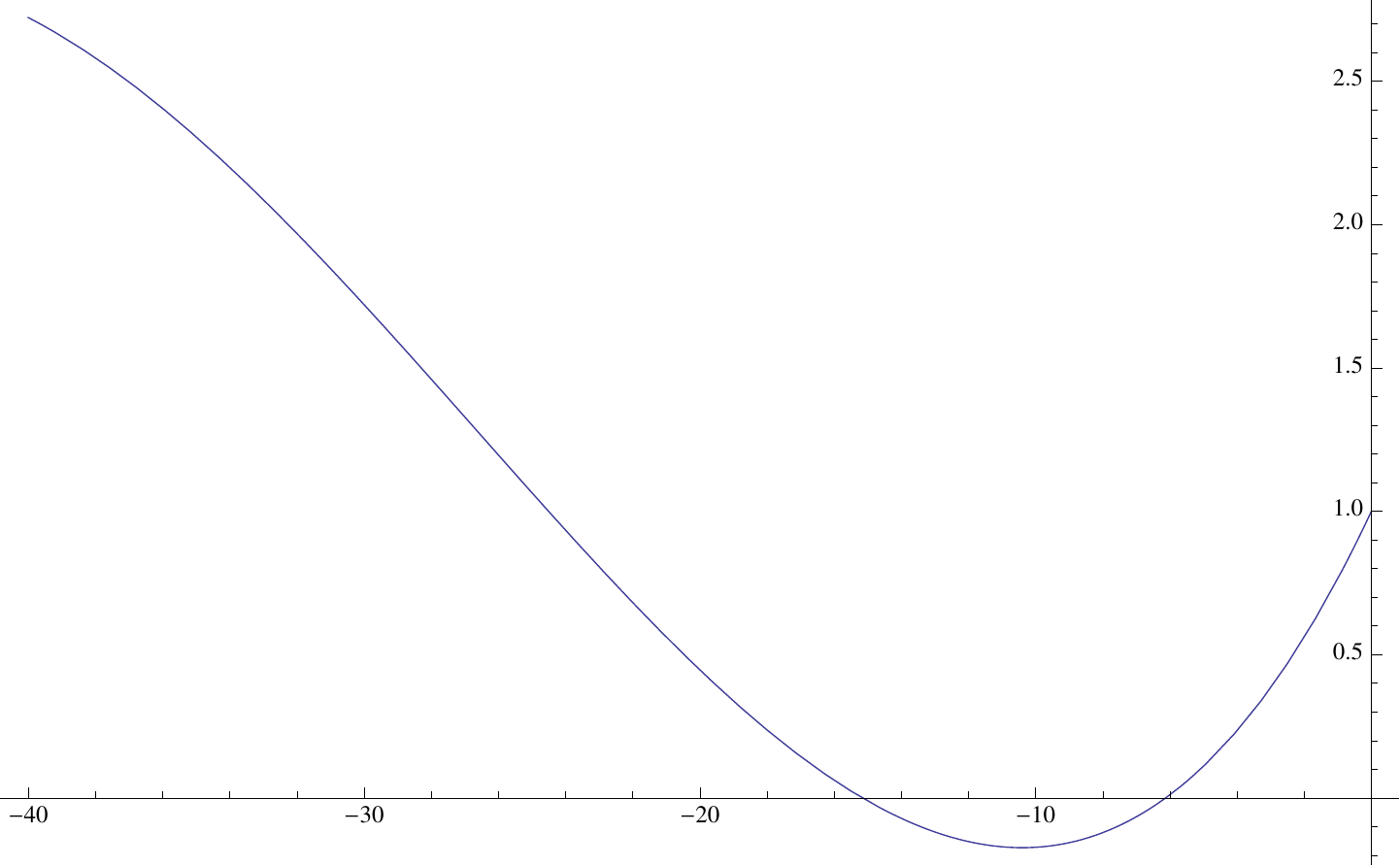} \hskip 1cm \includegraphics[scale=0.35]{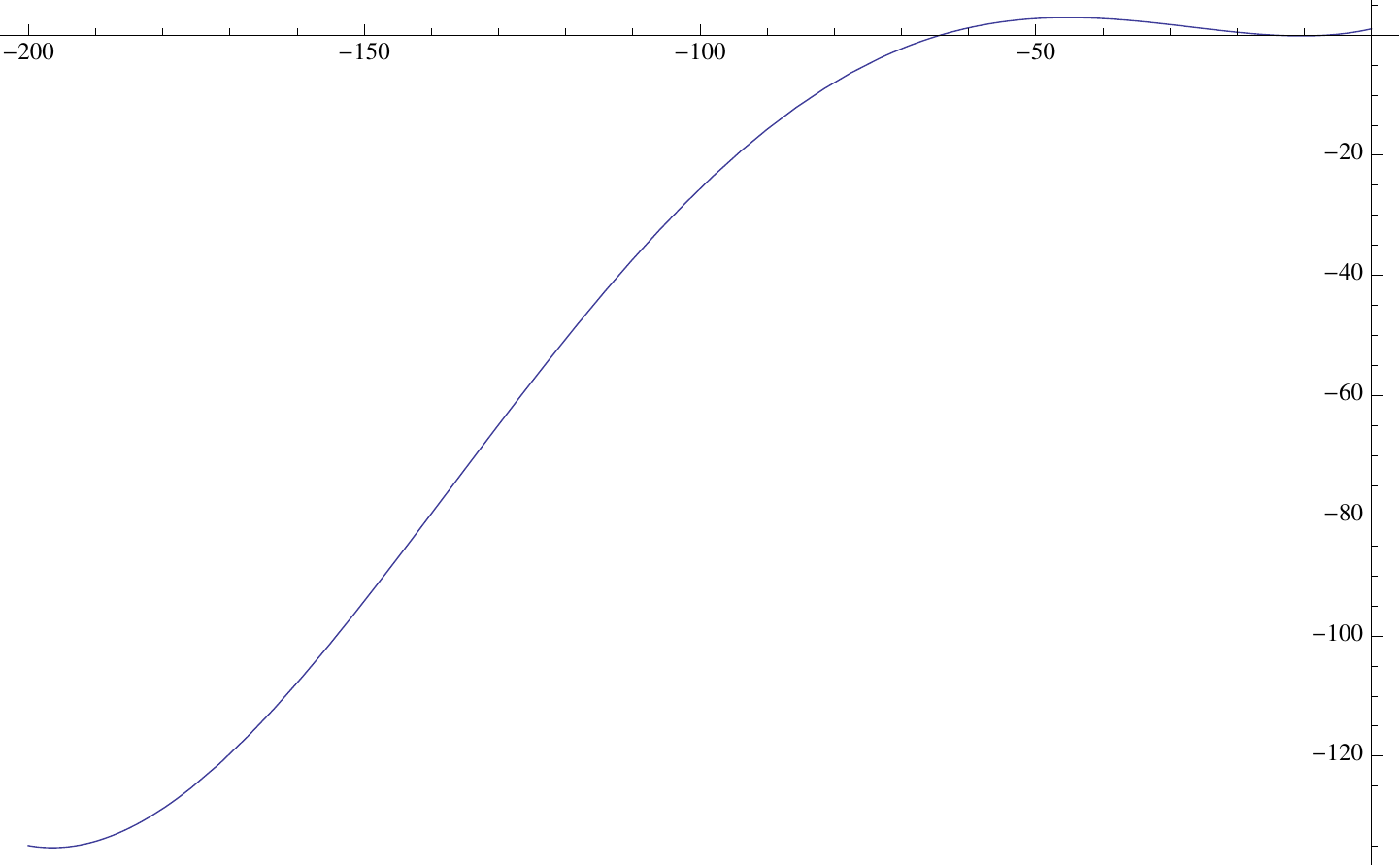}
\end{center}
\vskip 0.3cm

\caption {$\Psi(1/4,u)$ in the intervals $[-40,0]$ and $[-200,0]$.}
\label{fig1}
\end{figure}

Thus for any fixed $q\in (0,1)$  the  iteration scheme~\eqref{eq:rec} considered as a self-map of an appropriate space of analytic in an neighborhood of $[-1,0]$ functions has countably many fixed points.   At the moment  it is by no means clear what local properties these fixed points have. For example, for which $q$ and which of the above fixed points are repelling/attracting? Under which additional assumptions  on $f_1$ the sequence $\{f_j\}$ obtained via scheme~\eqref{eq:rec} converges? 

Our computer experiments and Theorem~~\ref{th:limit} suggest that the following statement should be true. 

\begin{conjecture} \label{th:iter} If $0<q\le \q$ and the initial function  of the form $f_1=(x+1)Q(x)$ where $Q(x)$ is a hyperbolic polynomial  with all negative roots smaller than $-1$, then the polynomial sequence $\{f_{j}(x)\}$ converges uniformly on $[-1,0]$ with any number of  derivatives to the function $\Psi(q,-u(q)x)$, where $u(q)$ is the negative solution of the equation $\Psi(q,u)=0$ with the minimal absolute value. 
\end{conjecture}

  \begin{remark} \rm { Apparently under the condition  $0<q\le \q$ the attraction domain of the latter fixed point is much larger than $f_1$ of the form given in the above conjecture. In particular, iterations started with  $f_1=\sin {\pi x}$ converge very quickly  to the same limit. On the other hand,  if $q >\q$ numeric experiments show that iterations typically diverge, see Fig~3.  It might be that for $q >\q\;$ all the fixed points of ~\eqref{eq:rec} become repelling.  There are superficial similarities of the scheme~\eqref{eq:rec}  and the famous logistic map in dynamical systems which also depends crucially on the value of the additional parameter $q$.}
  
  \end{remark}



 


\begin{figure}

\begin{center}
\includegraphics[scale=0.4]{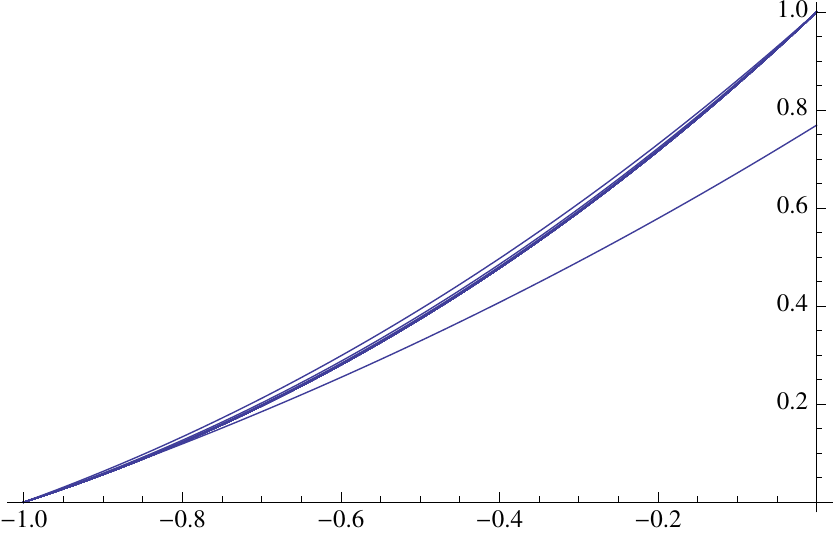} \hskip 1cm \includegraphics[scale=0.4]{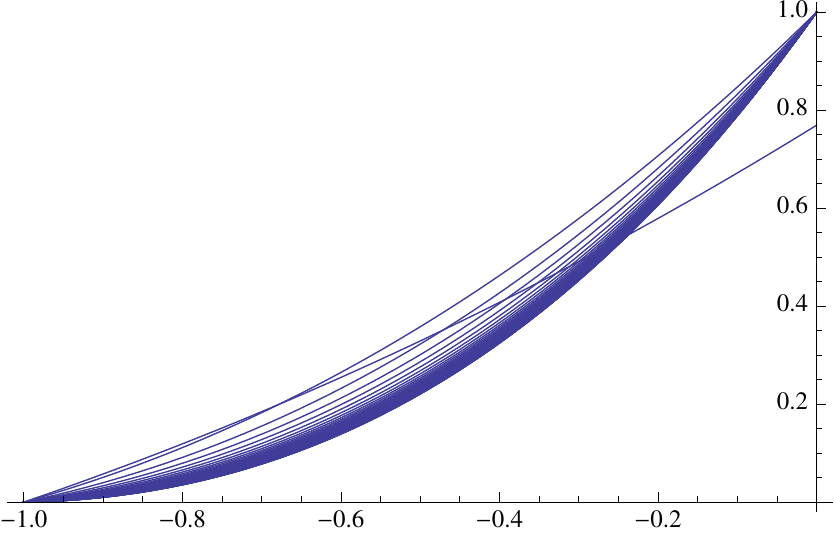}
\end{center}

\vskip 0.3cm

\caption {Convergence of iterations for $q=1/4$ (left) and the critical $\q$ (right).}
\label{fig4}
\end{figure}

  \begin{figure}

\begin{center}
\includegraphics[scale=0.4]{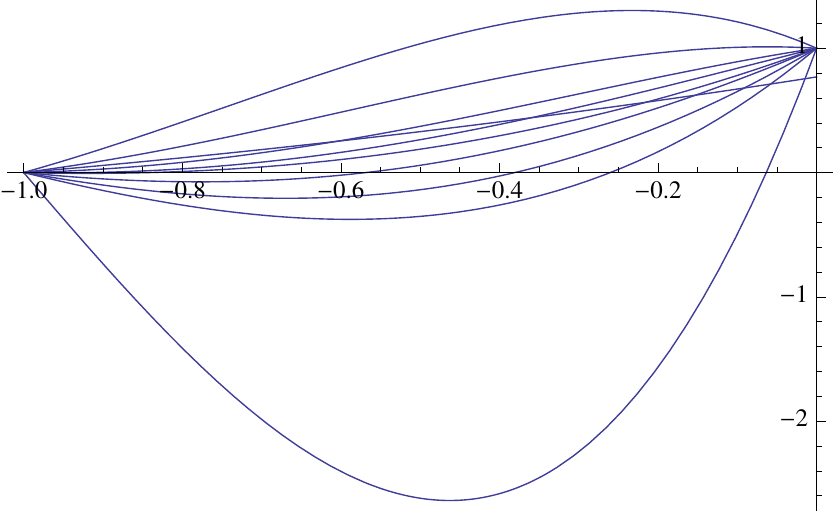} \hskip 1cm \includegraphics[scale=0.4]{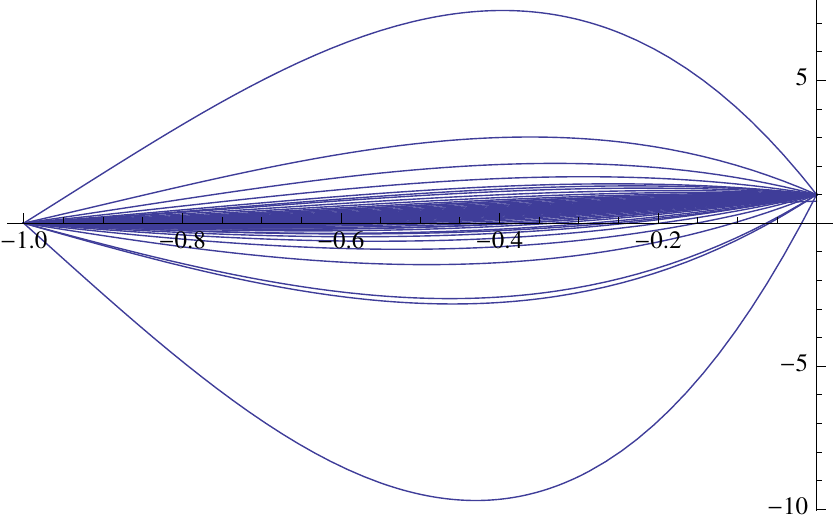}

\end{center}

\vskip 0.3cm

\caption {Divergence of iterations for $q=1/2$. The number of iterations on the left   is 10 and on the right is 50.}
\label{fig3}
\end{figure}

%

\end{document}